\documentclass[11pt, reqno]{amsart}
\usepackage{latexsym}

\usepackage{amsmath}
\usepackage{color}

\definecolor{blu}{rgb}{0,0,0.1}

\makeatletter
\def\itemize{
  \ifnum\@itemdepth>3\@toodeep\else
    \advance\@itemdepth\@ne
    \edef\@itemitem{labelitem\romannumeral\the\@itemdepth}%
        \list{\csname\@itemitem\endcsname}%
      {\leftmargin=20pt\def\makelabel##1{\hss\llap{##1}}}
        \fi}

\renewenvironment{enumerate}{%
  \ifnum \@enumdepth >3 \@toodeep\else
      \advance\@enumdepth \@ne
      \edef\@enumctr{enum\romannumeral\the\@enumdepth}\list
      {\csname label\@enumctr\endcsname}{\usecounter
        {\@enumctr}\leftmargin=20pt\def\makelabel##1{\hss\llap{\upshape##1}}}\fi
}{%
  \endlist
}

   \makeatletter
   \def\LaTeX{\leavevmode L\raise.42ex
       \hbox{\kern-.3em\size{\sf@size}{0pt}\selectfont A}\kern-.15em\TeX}
   \makeatother
   
   \newcommand{\BibTeX}{{\rm B\kern-.05em{\sc
             i\kern-.025emb}\kern-.08em\TeX}}
\def\bbm[#1]{\mbox{\boldmath $#1$}}



   \newcommand{\e }{\varepsilon }

   \renewcommand{\O }{\Omega }

   \newcommand{\intr }{\int_{\R^N}}
   \newcommand{\into }{\int_{\Omega}}

   \newcommand{\R}{{\mathbb{R}}}

\newcommand{\cal}{\mathcal }

   \newcommand{\beq}{\begin{equation}}
   \newcommand{\eeq}{\end{equation}}

   \renewcommand{\theequation}{\thesection.\arabic{equation}}

\numberwithin{equation}{section}

   \newtheorem{theorem}{Theorem}[section]
   
   \newtheorem{proposition}[theorem]{Proposition}
   \newtheorem{lemma}[theorem]{Lemma}
   
   \newtheorem{remark}[theorem]{Remark}
   \newcommand{\bremark}{\begin{remark} \em}
   \newcommand{\eremark}{\end{remark} }

\DeclareMathOperator{\dist}{dist}
\def\bbm[#1]{\mbox{\boldmath $#1$}}
\def\bbm[#1]{\mbox{\boldmath $#1$}}

\begin{document}

\title[Nodal solutions for slightly subcritical elliptic problems]
{Multi-bubble nodal solutions for slightly subcritical elliptic problems in domains with
symmetries}

\author{Thomas Bartsch \& Teresa D'Aprile \& Angela Pistoia}
\address{Thomas Bartsch, Mathematisches Institut, Justus-Liebig-Universit\"at Giessen,
Arndtstr. 2, 34392 Giessen, Germany.}
\email{Thomas.Bartsch@math.uni-giessen.de}
\address{Teresa D'Aprile, Dipartimento di Matematica, Universit\`a di Roma ``Tor
Vergata", via della Ricerca Scientifica 1, 00133 Roma, Italy.}
\email{daprile@mat.uniroma2.it }
\address{Angela Pistoia, Dipartimento di Metodi e Modelli Matematici, Universit\`a di Roma
``La Sapienza", via Antonio Scarpa 165, 00161 Roma, Italy.}
\email{pistoia@dmmm.uniroma1.it }

\thanks{T.B. has been supported by the Vigoni Project 50766047.}

\thanks{T.D. and A.P. have been supported by  the Italian PRIN Research Project 2009
\textit{Metodi variazionali e topologici nello studio dei fenomeni non lineari}}

\begin{abstract} We study the existence of sign-changing solutions  with multiple
\textit{bubbles} to the slightly subcritical problem
$$-\Delta u=|u|^{2^*-2-\e}u \,\hbox{ in }\Omega, \quad u=0 \,\hbox{ on }\partial \Omega,$$
where $\Omega$ is a smooth bounded domain in $\R^N$, $N\geq 3$, $2^*=\frac{2N}{N-2}$ and
$\e>0$ is a small parameter. In particular we prove that if $\Omega$ is convex and satisfies
a certain symmetry, then a nodal four-bubble solution exists with two positive and two
negative bubbles. \bigskip



\end{abstract}
\maketitle
\section{Introduction}
We are concerned with the slightly subcritical elliptic problem
\begin{equation}\label{eq1}
\left\{\begin{aligned}
&-\Delta u=|u|^{2^*-2-\e}u &\hbox{ in }&\Omega,\\
&u=0 &\hbox{ on }&\partial \Omega,
\end{aligned}\right.
\end{equation}
where $\Omega$ is a smooth and bounded domain in $\R^N$, $N\geq 3$, $\e>0$ is a small parameter.
Here $2^*$ denotes the critical exponent in the Sobolev embeddings, i.e. $2^*=\frac{2N}{N-2}$.

In \cite{poho} Poho\u{z}aev proved that the problem \eqref{eq1} does not admit a nontrivial
solution if $\Omega$ is star-shaped and $\e\leq0$. On the other hand problem \eqref{eq1} has
a positive solution if $\e\leq 0$ and $\Omega$ is an annulus, see Kazdan and Warner \cite{kaz}.
In \cite{baco} Bahri and Coron found a positive solution to  \eqref{eq1} with $\e=0$ provided
that the domain $\Omega$ has a \textit{nontrivial topology}. Moreover in \cite{delfemu1,delfemu2,delfemu3,pire} the authors considered the slightly
supercritical case where $\e<0$ is close to $0$ and proved solvability of \eqref{eq1} in
Coron's situation of a domain with one or more small holes.

In the subcritical case $\e>0$ the problem \eqref{eq1} is always solvable, since a positive
solution $u_\e$ can be found by solving the variational problem
$$\inf\left\{\into |\nabla u|^2\,\Big|\, u\in H^1_0(\Omega),\, \|u\|_{2^*-\e}=1\right\}.$$
In \cite{brepe,fluwe,han,rey1,rey2}  it was proved that, as
$\e\to 0^+$,  $u_\e$ blows up and concentrates at a point $\xi$ which is a critical point of
the Robin's function of $\Omega$. In addition to the one-peak solution $u_\e$, several papers
have studied concentration phenomena for positive solutions of \eqref{eq1} with multiple
blow-up points (\cite{balirey,rey0}). In a convex domain such a phenomenon
cannot occur. Grossi and Takahashi \cite{grotak} proved the nonexistence of positive
solutions for the problem \eqref{eq1} blowing up at more than one  point. On the other
hand, multi-peak nodal solutions always exist for problem \eqref{eq1} in a general bounded and
smooth domain $\Omega$. Indeed, in \cite{bamipi}  a solution with exactly one positive and one
negative blow-up point is constructed for the problem \eqref{eq1} if $\e>0$ is sufficiently
small. The location of the two  concentration points is also characterized and depends on the
geometry of the domain. Moreover the presence of sign-changing solutions with a multiple
blow-up at a single point has been proved in \cite{mupi1,piwe} for problem
\eqref{eq1}; such solutions have the shape of towers of alternating-sign bubbles, i.e. they
are superpositions of positive bubbles and negative bubbles blowing-up at the same point with
a different velocity.  We also quote the paper \cite{bep}, where the authors  study   the blow up of the low energy sign-changing solutions of problem \eqref{eq1}
and they   classify these solutions according
to the concentration speeds of the positive and negative part.
Finally, we   mention the papers \cite{ba} and \cite{bawe} where, by a
different approach, the authors provide existence and multiplicity of sign-changing
solutions for more general problems than \eqref{eq1}. These papers are however not concerned
with the profile of the solutions.

In this paper we deal with the construction of sign-changing solutions which develop a
spike-shape as $\e\to 0^+$, blowing up positively at some points and negatively at other
points, generalizing the double blowing up obtained in \cite{bamipi}. We are able to prove
that on certain domains $\O$, \eqref{eq1} admits solutions with exactly two positive and two
negative blow-up points. Moreover, the asymptotic profile of the blow-up of these solutions
resembles a \textit{bubble}, namely a solution of the equation at the critical exponent in
the entire $\R^N$. It is natural to ask about the existence of solutions with $k$ blow-up
points, also for $k\ne 2,4$, and in more general domains. We shall discuss this difficult
problem below.

In order to formulate the conditions on the domain $\O$, we need to introduce some notation.
Let us denote by $G(x,y)$ the Green's function of $-\Delta$ over $\Omega$ under Dirichlet
boundary conditions; so $G$ satisfies
\begin{equation*}
\left\{\begin{aligned}
&-\Delta_yG(x,y)=\delta_x(y) &\hbox{ }&y\in\Omega,\\
&G(x,y)=0 &\hbox{ }&y\in\partial \Omega,
\end{aligned}\right.
\end{equation*}
where $\delta_x$ is the Dirac mass at $x$. We denote by $H(x,y)$ its regular part, namely
$$H(x,y)=\frac{1}{(N-2)\sigma_N|x-y|^{N-2}}-G(x,y),$$
where $\sigma_N$ is the surface measure of the unit sphere in $\R^N$. The diagonal $H(x,x)$
is called the Robin's function of the domain $\Omega$.

Here are our assumptions on $\O$.
\begin{enumerate}
\item[(A1)] $\Omega\subset \R^N$, $N\geq 3$, is a bounded domain with a ${\cal C}^2$-boundary.
\item[(A2)] $\O$ is invariant under the reflection $(x_1,x')\mapsto(x_1,-x')$ where
$x_1\in\R$, $x'\in\R^{N-1}$.
\end{enumerate}
For simplicity of notation we write the restrictions of $G$ and $H$ to the $x_1$-axis as $g$
and $h$ respectively, i.e.
$$
g(t,s)=G((t,0,\ldots, 0), (s,0,\ldots, 0))\quad\text{ and }\quad
h(t,s)=H((t,0,\ldots, 0), (s,0,\ldots, 0)).
$$
Our last assumption concerning the domain is:
\begin{enumerate}
\item[(A3)] There exists a connected component $(a,b)$ of the set
$\{t\,|\,(t,0,\ldots, 0)\in \Omega\}\subset \R$ such that
\begin{equation}\label{prop1}
\hbox{ the function } (a,b)\ni t \mapsto h(t, t)\hbox{ is convex}\end{equation}
and
\begin{equation}\label{prop2}
\hbox{ for any }t,s\in (a,b),\,t\neq s : \; (t-s)\frac{\partial g}{\partial t}(t,s)<0.
\end{equation}
\end{enumerate}

We can now state our main result.

\begin{theorem}\label{th1}
If $\O$ satisfies (A1), (A2), (A3), then for $\e>0$ sufficiently small problem \eqref{eq1} has
a solution $u_\e$ with the following property. There exist numbers $\lambda_i^\e>0$ and points
$\xi_i^\e=(t_i^\e,0,\ldots, 0)\in\Omega$ with $t_i^\e\in(a,b)$, $i=1,2,3,4$, such that
$$
u_\e(x)
=\alpha_N\sum_{i=1}^4(-1)^{i+1}\bigg(
  \frac{\lambda_{i}^\e\e^{\frac{1}{N-2}}}{\e^{\frac{2}{N-2}}\lambda_i^\e+|x-\xi_i^\e|^2}
  \bigg)^{\frac{N-2}{2}}+o(1)
\hbox{ uniformly in }\overline{\Omega};
$$
here $\alpha_N=(N(N-2))^{(N-2)/4}$. Moreover, the numbers $\lambda_i^\e$ are bounded above and
below away from zero, and the numbers $t_i^\e$ are aligned on $(a,b)$ and remain uniformly
away from the boundary and from one another, i.e.
$$\delta<\lambda_i^\e<\frac{1}{\delta}\quad \forall i=1,2,3,4,$$
and
$$a+\delta<t_1^\e<t_2^\e<t_3^\e<t_4^\e<b-\delta, \quad t_{i+1}^\e-t_i^\e>\delta\quad
 \forall i=1,2,3,$$ for some $\delta>0.$
\end{theorem}

Let us observe that the assumption (A3) is satisfied for a (not necessarily strictly) convex
domain $\Omega$ as a consequence of some properties of the Green's and the Robin's functions.
Indeed, \eqref{prop1} follows from the result in \cite{cara} according to which the Robin's
function of a convex domain is strictly convex. Moreover in a convex domain the function
$G(\cdot, y)$ is strictly decreasing (with non-zero derivative) along the half-lines starting
from $y$ (see Lemma \ref{lem}), hence \eqref{prop2} holds true. Assumption (A3) is also
satisfied for some non-convex domains, for instance those which are ${\cal C}^2$-close to
convex domains. It seems to be an open problem whether (A3) holds, for instance, on annuli.

The proof  of Theorem \ref{th1} relies on a Lyapunov-Schmidt reduction scheme. This reduces
the problem of finding multi-bubble solutions for \eqref{eq1} to the problem of finding
critical points of a functional which depends on points $\xi_i$ and scaling parameters
$\lambda_i$. The leading part of the reduced functional is explicitly given in terms of the
Green's and Robin's functions. The reduced functional has a quite involved behaviour, due to
the different interactions among the bubbles (which depends on their respective sign). The
symmetry of the domain plays a crucial role: indeed, the validity of the hypothesis (A2)
allows us to place the positive and negative bubbles alternating along the one-dimensional
interval $(a,b)$. Then we use a variational approach and we
obtain the existence of a saddle point by applying a max-min argument. An important step
is the proof of a compactness condition which ensures that the max-min level actually is a
critical value, and this is the most technical and difficult part of the proof.

As remarked above, it is natural to ask about other types of multibump solutions, and to
consider more general domains. First of all, the Lyapunov-Schmidt reduction scheme works in a
very general setting. In particular, (A2) and (A3) are not required for this. The problem lies
in finding critical points of the reduced functional. This problem seem to be very subtle. In
the paper \cite{badapi1} we consider the case of a ball and we show the existence of two
three-bubble solutions having different nodal properties. However, these solutions are not
found via a global variational argument and the proof strongly depends on the explicit formula
of the Green's and the Robin's function in a ball. It also seems very hard to weaken the
assumptions on the domain. In our argument we use the symmetry condition (A2) in order to
localize and order the peaks on the $x_1$-axis. Together with (A3) this allows comparison
arguments involving the Green's and Robin's functions which do not hold in general.

The paper is organized as follows. In Section 2 we sketch the finite-dimensional reduction
method. Section 3 is devoted to solving the reduced problem by the max-min procedure. Finally
in the Appendix A we collect some properties of the Green's function which are usually
referred to throughout the paper.

\section{The reduced functional}
The proof of Theorem \ref{th1}  is based on the \textit{finite dimensional reduction}
procedure which has been used for a wide class of singularly perturbed problems. We sketch
the procedure here and refer to  \cite{bamipi} for details. Related methods have been
developed in  \cite{delfemu1}-\cite{delfemu2}-\cite{delfemu3} where the almost critical
problem \eqref{eq1} was studied from the supercritical side. In this section the assumptions
(A2) and (A3) are not required.

For any $\e>0$ let us introduce the functions
$$
U_{\e,\lambda, \xi}(x)
 =\alpha_N\bigg(
   \frac{\lambda\e^{\frac{1}{N-2}}}{\lambda^2\e^{\frac{2}{N-2}}+|x-\xi|^2}\bigg)^{\frac{N-2}{2}},
\quad \alpha_N=(N(N-2))^{(N-2)/4},
$$
with $\lambda>0$ and $\xi\in\R^N$. These are actually all positive solutions of the limiting
equation
$$
-\Delta U=U^{2^*-1}\hbox{ in }\R^N,
$$
and constitute the extremals for the Sobolev's critical embedding (see \cite{au},
\cite{cagispru}, \cite{tal}). Fixing $k\geq 1$, we define the configuration space
$$
\cal{O}_{k} :=
\left\{({ \bbm[\lambda]},\bbm[\xi])
      =( \lambda_1, \ldots,\lambda_k,\xi_1,\ldots,\xi_k)\,\Bigg|\,
\begin{aligned}
&\delta<\lambda_i<\delta^{-1},\; \xi_i\in\Omega,\;
  {\rm{dist}}(\xi_i,\partial\Omega)>\delta \;\;\forall i\\
&|\xi_i-\xi_j|>\delta\,\hbox{ if }i\neq j
\end{aligned}
\right\}
$$
where $\delta>0$ is a sufficiently small number. For fixed integers
$a_1,\ldots, a_k\in\{-1,1\}$, we seek suitable scalars $\lambda_i$ and points $\xi_i$ such
that a solution $u$ exists for \eqref{eq1} with $u\approx \sum_{i=1}^ka_i U_{\e,\lambda_i,\xi_i}$.
In order to obtain a better first approximation, which satisfies the boundary condition, we
consider the projections ${\cal P}_\Omega U_{\e,\lambda, \xi}$ onto the space $H^1_0(\Omega)$ of
$U_{\e,\lambda, \xi}$, where the projection  ${\cal P}_\Omega:H^1(\R^N)\to H^1_0(\Omega)$ is
defined as the unique solution of the problem
$$
\left\{
\begin{aligned}
&\Delta {\cal P}_\Omega u=\Delta u &\hbox{ in } \Omega,\\
& {\cal P}_\Omega u=0&\hbox{ on }\partial \Omega.
\end{aligned}\right.$$
Then the following estimate holds
\begin{equation}\label{known}
{\cal P}_\Omega U_{\e,\lambda_i, \xi_i}= U_{\e,\lambda_i, \xi_i}+O(\sqrt{\e})
\end{equation}
uniformly with respect to $(\bbm[\lambda],\bbm[\xi])\in {\cal O}_k$. We look for a solution
to \eqref{eq1} in a small neighbourhood of the first approximation, i.e. a solution of the form
$$
u:=\sum_{i=1}^k a_i{\cal P}_\Omega  U_{\e,\lambda_i,\xi_i}+\phi,
$$
where the rest term $\phi$ is small. To carry out the construction of a solution of this type,
we first introduce an intermediate problem as follows.

We consider the spaces
$$
{\cal K}_{\e,\hbox{\scriptsize$\bbm[\lambda]$},\hbox{\scriptsize$\bbm[\xi]$}}
 ={\rm{span}}\left\{
   {\cal P}_\Omega \bigg(\frac{\partial U_{\e,\lambda_i, \xi_i}}{\partial \xi_i^j}\bigg),
   {\cal P}_\Omega \bigg(\frac{\partial U_{\e,\lambda, \xi}}{\partial \lambda_i}\bigg)\,\bigg|\,
   i=1,\ldots,k,\;j=1,\ldots,N\right\}
\subset H^1_0(\Omega),
$$
and
$$
{\cal K}^\perp_{\e,\hbox{\scriptsize$\bbm[\lambda]$},\hbox{\scriptsize$\bbm[\xi]$}}
=\left\{\phi\in H^1_0(\Omega)\,\bigg|\, \langle \phi,
   \psi\rangle:=\into \nabla \phi\nabla \psi=0\;\;
   \forall \psi\in {\cal K}_{\e,\hbox{\scriptsize$\bbm[\lambda]$},\hbox{\scriptsize$\bbm[\xi]$}}\right\}
\subset H^1_0(\Omega);
$$
here we denote by $\xi_i^j$ the $j$-th component of $\xi_i$. Then it is convenient to solve
as a first step the problem for $\phi$ as a function of $\e$, $\bbm[\lambda]$, $\bbm[\xi]$.
This turns out to be solvable for any choice of points $\xi_i$ and scalars $\lambda_i$,
provided that $\e$ is sufficiently small. The following result was established in \cite{bamipi}.

\begin{lemma}\label{reg}
There exists $\e_0>0$ and a constant $C>0$ such that for each $\e\in (0,\e_0)$ and each
$(\bbm[\lambda],\bbm[\xi])\in {\cal O}_k$ there exists a unique
$\phi_{\e,\hbox{\scriptsize$\bbm[\lambda]$},\hbox{\scriptsize$\bbm[\xi]$} }
 \in {\mathcal{K}}_{\e,\hbox{\scriptsize$\bbm[\lambda]$},\hbox{\scriptsize$\bbm[\xi]$}}^\perp$ satisfying
\begin{equation}\label{sati1}
\Delta(V_{\e,\hbox{\scriptsize$\bbm[\lambda]$},\hbox{\scriptsize$\bbm[\xi]$}}+\phi)
 +|V_{\e,\hbox{\scriptsize$\bbm[\lambda]$},\hbox{\scriptsize$\bbm[\xi]$}}+\phi|^{2^*-2-\e}
  (V_{\e,\hbox{\scriptsize$\bbm[\lambda]$},\hbox{\scriptsize$\bbm[\xi]$}}+\phi)
\in {\cal K}_{\e,\hbox{\scriptsize$\bbm[\lambda]$},\hbox{\scriptsize$\bbm[\xi]$}}
\end{equation}
and
\begin{equation}\label{sati2}
\|\phi\|:=\Big(\into|\nabla \phi|^2\Big)^{1/2}< C\e.
\end{equation}
Here
$V_{\e,\hbox{\scriptsize$\bbm[\lambda]$},\hbox{\scriptsize$\bbm[\xi]$} }
 =\sum_{i=1}^k a_i{\cal P}_\Omega  U_{\e,\lambda_i,\xi_i}$.
Moreover the map ${\cal O}_k \to H^1_0(\Omega)$,
$(\bbm[\lambda],\bbm[\xi]) \mapsto \phi_{\e,\hbox{\scriptsize$\bbm[\lambda]$},\hbox{\scriptsize$\bbm[\xi]$} }$
is ${\mathcal C}^1$.
\end{lemma}

After this result, let us consider  the following energy functional associated with problem
\eqref{eq1}:
\begin{equation}\label{func1}
I_\e(u)=\frac12\into |\nabla u|^2 dx-\frac{1}{2^*-\e}\into |u|^{2^*-\e} dx ,\quad u\in H^1_0(\Omega).
\end{equation}
Solutions of \eqref{eq1} correspond to critical points of $I_\e$. Now we introduce the new
functional
\begin{equation}\label{jeps}
J_\e:{\cal O}_k\to \R, \quad J_\e(\bbm[\lambda],\bbm[ \xi])
 =I_\e(V_{\e,\hbox{\scriptsize$\bbm[\lambda]$},\hbox{\scriptsize$\bbm[\xi]$} }
       +\phi_{\e,\hbox{\scriptsize$\bbm[\lambda]$},\hbox{\scriptsize$\bbm[\xi]$} })
\end{equation}
where $\phi_{\e,\hbox{\scriptsize$\bbm[\lambda]$},\hbox{\scriptsize$\bbm[\xi]$}}$ has been constructed in
Lemma \ref{reg}. The next lemma has been proved in \cite{balirey} and reduces the original
problem \eqref{eq1} to the one of finding critical points of the functional $J_\e$.

\begin{lemma}\label{relation}
The pair $(\bbm[\lambda],\bbm[\xi])  \in {\cal O}_k$ is a critical point of $J_\e $ if and
only if the corresponding function
$u_\e=V_{\e,\hbox{\scriptsize$\bbm[\lambda]$},\hbox{\scriptsize$\bbm[\xi]$} }
       +\phi_{\e,\hbox{\scriptsize$\bbm[\lambda]$},\hbox{\scriptsize$\bbm[\xi]$} }$
is a solution of (\ref{eq1}).
\end{lemma}

Finally we describe an expansion for $J_\e$ which can be obtained as in
\cite{delfemu2}-\cite{delfemu3}.

\begin{proposition}\label{exp1}
With the change of variables $\lambda_i=(c_N\Lambda_i)^{\frac{1}{N-2}}$ the following asymptotic
expansion holds:
\begin{equation}\label{lonel}
J_\e (\bbm[\lambda],\bbm[\xi] )
= kC_N+\frac{k}{2}\omega_N\e\log\e+k\gamma_N\e+\omega_N\e\Psi_k(\bbm[\Lambda],\bbm[\xi])+o(\e)
\end{equation}
$\cal C^1$-uniformly with respect to  $(\bbm[\lambda],\bbm[\xi])\in{\cal O}_k$. Here:
$$\Psi_k( \bbm[\Lambda],\bbm[\xi])
=\frac12\sum_{i=1}^k\Lambda_i^2H(\xi_i,\xi_i)-\sum_{i<j}a_ia_j\Lambda_i\Lambda_jG(\xi_i,\xi_j)
 -\log (\Lambda_1\cdot\ldots\cdot\Lambda_k),$$
and, setting $U=U_{1,1,0}$, the constants $C_N,\,c_N,\, \omega_N,$ and $\gamma_N$ are given by
$$ C_N=\intr|\nabla U|^2-\frac{1}{2^*}\intr U^{2^*},\quad
 c_N=\frac{1}{2^*}\frac{\intr U^{2^*}}{(\intr U^{2^*-1})^2},\quad
 \omega_N=\frac{1}{2^*}\intr U^{2^*},$$
and
$$
\gamma_N=\frac{1}{(2^*)^2}\intr U^{2^*}-\frac{1}{2^*}\intr U^{2^*}\log U +\frac12\omega_N\log c_N.
$$
\end{proposition}

Thus in order to construct a solution of problem \eqref{eq1} such as the one predicted in
Theorem \ref{th1} it remains to find a critical point of $J_\e$. This will be accomplished in
the next two sections.

We finish this section with a symmetry property of the reduction process.

\begin{lemma}\label{symmetry}
Suppose $\O$ is invariant under the action of an orthogonal transformation $T\in O(N)$. Let
${\cal O}_k^T:=\{(\bbm[\Lambda],\bbm[\xi])\in{\cal O}_k: T\xi_i=\xi_i\ \ \forall i\}$ denote
the fixed point set of $T$ in ${\cal O}_k$. Then a point
$(\bbm[\Lambda],\bbm[\xi])\in{\cal O}_k^T$ is a critical point of $J_\e$ if it is a
critical point of the constrained functional $J_\e|{\cal O}_k^T$.
\end{lemma}

\begin{proof}
We first investigate the symmetry inherited by the function
$\phi_{\e,\hbox{\scriptsize$\bbm[\lambda]$},\hbox{\scriptsize$\bbm[\xi]$}}$ obtained in Lemma \ref{reg}.
Setting $T\bbm[\xi]:=(T\xi_1,\ldots,T\xi_k)$ for $\bbm[\xi]=(\xi_1,\ldots,\xi_k)\in\O^k$,
we claim that
\begin{equation}\label{cell}
\phi_{\e,\hbox{\scriptsize$\bbm[\lambda]$},\hbox{\scriptsize$\bbm[\xi]$}}
 =\phi_{\e,\hbox{\scriptsize$\bbm[\lambda]$},T\hbox{\scriptsize$\bbm[\xi]$}}\circ T\quad
 \forall (\bbm[\lambda],\bbm[\xi] )\in {\cal O}_k.
\end{equation}
Indeed, because of the symmetry of the domain, we see that
$$
{\cal P}_\Omega U_{\e,\lambda_i, \xi_i} = ({\cal P}_\Omega U_{\e,\lambda_i, T\xi_i})\circ T
$$
and
$$
{\cal K}_{\e,\hbox{\scriptsize$\bbm[\lambda]$},\hbox{\scriptsize$\bbm[\xi]$}}
 =\{f\circ T\,|\, f\in {\cal K}_{\e,\hbox{\scriptsize$\bbm[\lambda]$},T\hbox{\scriptsize$\bbm[\xi]$}}\},\qquad
{\cal K}_{\e,\hbox{\scriptsize$\bbm[\lambda]$},\hbox{\scriptsize$\bbm[\xi]$}}^\perp
 =\{f\circ T\,|\, f\in {\cal K}^\perp_{\e,\hbox{\scriptsize$\bbm[\lambda]$},T\hbox{\scriptsize$\bbm[\xi]$}}\}.
$$
Then the function  $\phi_{\e,\hbox{\scriptsize$\bbm[\lambda]$},T\hbox{\scriptsize$\bbm[\xi]$}}\circ T$ belongs to
$ {\cal K}_{\e,\hbox{\scriptsize$\bbm[\lambda]$},\hbox{\scriptsize$\bbm[\xi]$} }^\perp$ and satisfies
\eqref{sati1} and \eqref{sati2}. The uniqueness of the solution $\phi$ implies \eqref{cell}.
Therefore the functional $J_\e$ satisfies
$$
J_\e(\bbm[\lambda],\bbm[ \xi])=J_\e(\bbm[\lambda], T\bbm[\xi]).
$$
The lemma follows immediately.
\end{proof}

\section{A max-min argument: proof of Theorem \ref{th1}}
In this section we will employ the reduction approach to construct the solutions stated in
Theorem \ref{th1}. The results obtained in the previous section imply that our problem reduces
to the study of critical points of the functional $J_\e$ defined in \eqref{jeps}. In what
follows, we assume (A1), (A2), (A3). For $t_1,\ldots,t_k\in (a,b)$, where $(a,b)$ is from (A3),
we set $\bbm[t]=(t_1,\ldots, t_k)$ and
$$
\tilde{J_\e}(\bbm[\lambda],\bbm[t])
 =J_\e(\bbm[\lambda],(t_1,0,\ldots,0),(t_2,0,\ldots,0),\ldots,(t_k,0,\ldots,0)).
$$

\begin{lemma}\label{validity}
If $(\bbm[\lambda], \bbm[t])$ is a critical point of $\tilde{J}_\e$, then
$(\bbm[\lambda], \bbm[\xi])$ is a critical point of $J_\e$, where $\xi_i=(t_i,0,\ldots,0)$.
\end{lemma}

\begin{proof} This is an immediate consequence of Lemma \ref{symmetry}.
\end{proof}

Let us now fix $k=4$ and set $$a_1=a_3=1,\quad a_2=a_4=-1.$$  So we are looking for solutions
to problem \eqref{eq1} with 2 positive and two negative spikes which are aligned along the
$x_1$-direction with alternating signs. From Lemma \ref{validity}, we need to find a critical
point of the function $\tilde{J}_\e(\bbm[\lambda], \bbm[t])$. The expansion obtained in
Proposition \ref{exp1} implies that our problem reduces to the study of critical points of a
functional which is a small ${\cal C}^1$-perturbation of
$$
\tilde{\Psi}( \bbm[\Lambda],\bbm[t])
 =\frac12\sum_{i=1}^4\Lambda_i^2h(t_i, t_i)-\sum_{i<j}(-1)^{i+j}\Lambda_i\Lambda_jg(t_i,t_j)
  -\log (\Lambda_1\cdot\Lambda_2\cdot\Lambda_3\cdot\Lambda_4),
$$
where
$\bbm[\Lambda]
 =(\Lambda_1,\Lambda_2,\Lambda_3,\Lambda_4)\in(0,+\infty)^4$, $\bbm[t]
 =(t_1,t_2,t_3,t_4)\in(a,b)^4$
and the functions  $g$ and $h$ are the restrictions of $G$ and $H$ to the $x_1$-axis defined
in the introduction. We recall that the function $\tilde{\Psi}$ is well defined in the set
$$
{\cal M} :=
 \Big\{(\bbm[\Lambda],\bbm[t])\,\Big|\,\Lambda_i>0,\,t_i\in(a,b)\;
 \forall i=1,2,3,4 \;\;\&\;\; t_1<t_2<t_3<t_4\Big\}.
$$

Observe that  by assumption \eqref{prop2} the function $g(\cdot, s)=g(s, \cdot)$ is decreasing
along the interval $(s, b)$ and increasing along $(a,s)$. Therefore
\begin{equation}\label{decre1}
g(t_1, t_4)\leq g(t_1, t_3)\leq g(t_1, t_2), \qquad g(t_1, t_4)\leq g(t_2, t_4)\leq g(t_3, t_4)
 \quad \forall (\bbm[\Lambda],\bbm[t])\in{\cal M} .
\end{equation}
Analogously,
\begin{equation}\label{decre2}
g(t_2, t_4),\, g(t_1, t_3)\leq g(t_2, t_3)\quad \forall (\bbm[\Lambda],\bbm[t])\in{\cal M}.
\end{equation}
In this section we apply a max-min argument to characterize a topologically nontrivial
critical value of the function $\tilde\Psi$ in the set ${\cal M}$. More precisely we will
construct sets $\mathcal D$, $K$, $K_0\subset{\cal M}$ satisfying  the following properties:
\begin{enumerate}
\item [(P1)] $\mathcal D$ is an open set,  $K_0$ and $K$ are
compact sets,  $K$ is connected and
\begin{equation*}
K_0\subset K\subset \mathcal D\subset {\overline {\mathcal D}}\subset {\cal M};
\end{equation*}
\item[(P2)] If  we define the complete metric space ${\mathcal F}$ by
$$
{\mathcal F}=\Big\{\eta:K\to {\mathcal D}\,\Big|\,
 \eta\hbox{ continuous},\;\eta(\bbm[\Lambda],\bbm[t])=(\bbm[\Lambda],\bbm[t])\;
 \forall (\bbm[\Lambda],\bbm[t])\in K_0\Big\},
$$
then
\begin{equation}\label{mima}
\tilde{ \Psi}^*
:=\sup_{\eta\in{\mathcal F}}\min_{(\hbox{\scriptsize$\bbm[\Lambda]$},\hbox{\scriptsize$\bbm[t]$)}\in K}
  \tilde{\Psi}(\eta(\bbm[\Lambda],\bbm[t]))
<\min_{(\hbox{\scriptsize$\bbm[\Lambda]$},\hbox{\scriptsize$\bbm[t]$})\in K_0}\tilde{\Psi}(\bbm[\Lambda],\bbm[t]).
\end{equation}
\item[(P3)] For every $(\bbm[\Lambda],\bbm[t])\in\partial\mathcal D$ such that
$\tilde{\Psi}(\bbm[\Lambda],\bbm[t])={\Psi}^*$, we have that $\partial \mathcal D$ is smooth at
$(\bbm[\Lambda],\bbm[t])$ and there exists a vector
$\tau_{\hbox{\scriptsize$\bbm[\Lambda]$},\hbox{\scriptsize$\bbm[t]$}}$ tangent to $\partial\mathcal D$ at
$(\bbm[\Lambda],\bbm[t])$ so that
$\tau_{\hbox{\scriptsize$\bbm[\Lambda]$},\hbox{\scriptsize$\bbm[t]$}} \cdot
 \nabla \tilde\Psi(\bbm[\Lambda],\bbm[t])\neq 0$.
\end{enumerate}

\bigskip

Under these assumptions a critical point $(\bbm[\Lambda],\bbm[t])\in{\mathcal D}$ of
$\tilde\Psi$ with $\tilde{\Psi}(\bbm[\Lambda],\bbm[t])=\tilde{\Psi}^*$ exists, as a standard
deformation argument involving the gradient flow of $\tilde{\Psi}$ shows. Moreover, since
properties (P2)-(P3) continue to hold also for a function which is ${\cal C}^1$-close to
$\tilde{\Psi}$, then such a critical point will \textit{survive} small
${\cal C}^1$-perturbations.

\medskip

\subsection{Definition of $\bbm[{\cal D}]$} We define
$$
{\mathcal D}
 =\bigg\{(\bbm[\Lambda],\bbm[t])\in{\cal M}\;\bigg|\,
  \Phi(\bbm[\Lambda],\bbm[t])
    :=\frac12\sum_{i=1}^4\Lambda_i^2h(t_i,t_i)+\sum_{i<j}\Lambda_i\Lambda_jg(t_i,t_j)
 -\log (\Lambda_1\Lambda_2\Lambda_3\Lambda_4) <M \bigg\}
$$
where $M>0$ is a sufficiently large number to be specified later. It is easy to check  that
the function $\Phi$ satisfies
\begin{equation}\label{coercivi}
\Phi(\bbm[\Lambda],\bbm[t])\to +\infty\hbox{ as }(\bbm[\Lambda],\bbm[t])\to \partial{\cal M}.
\end{equation}
Indeed, for any $\Lambda>0$ and $t\in (a,b)$ we have
$$
\frac{\Lambda^2}{2} h(t,t)-\log\Lambda
 \geq\frac{\Lambda^2}{4}h(t,t)+|\log\Lambda|+\Big(\frac{\Lambda^2}{4}H_0-2\log^+ \Lambda \Big)
$$
where $\log^+x=\max\{\log x, 0\}$  denotes the positive  part of the logarithm,  and $H_0>0$
is the minimum value of the Robin's function in $\Omega$ (see \eqref{coercive}). Taking into
account that the function $\frac{H_0}{4}x^2 -2\log x$ minimizes for $x=2H_0^{-1/2}$, we deduce
\begin{equation}\label{coecoe}
\frac{\Lambda^2}{2} h(t,t)-\log\Lambda
 \geq \frac{\Lambda^2}{4}h(t,t)+|\log \Lambda|-2\log^+\frac{2}{\sqrt{H_0}}\quad
 \forall \Lambda>0,\, t\in (a,b).
\end{equation}
Hence for any  $(\bbm[\Lambda], \bbm[t])\in {\cal M}$ we get
\begin{equation}\label{tele}
\Phi(\bbm[\Lambda],\bbm[t])
 \geq \frac14\sum_{i=1}^4\Lambda_i^2h(t_i,t_i)+\sum_{i=1}^4|\log\Lambda_i|
       +\sum_{i<j}\Lambda_i\Lambda_jg(t_i,t_j)-8\log^+\frac{2}{\sqrt{H_0}}.
\end{equation}
\eqref{coercivi} follows by using the properties of $h$ and $g$ (see  Appendix A). In
particular \eqref{coercivi} implies that ${\mathcal D}$ is compactly contained in ${\cal M}$.

\medskip

\subsection{Definition of $\bbm[K]$, $\bbm[K_0]$, and proof of (P1)}

In this subsection we define the sets $K,\,K_0$ for which properties (P1)-(P2) hold. We
consider the configurations $(\bbm[\Lambda], \bbm[t])$ such that $\Lambda_2=\Lambda_3$,
i.e. configurations of the form
\begin{equation}\label{firstform}
(\bbm[\Lambda](\bbm[\mu]),\bbm[t])
 =\left(\frac{\mu_1}{\sqrt{\mu}},\sqrt{\mu},\sqrt{\mu},\frac{\mu_4}{\sqrt{\mu}},
   t_1,t_2,t_3,t_4\right),
\end{equation}
where $\bbm[t]=(t_1,t_2,t_3,t_4)\in (a,b)^4$, and
${\bbm[\mu]}=(\mu_1,\mu, \mu_4)\in (0,+\infty)^3$.
Next we consider the open set
\begin{equation}\label{utilde}
\bigg\{(\bbm[\mu],\bbm[t])\in (0,+\infty)^3\times (a,b)^4\,\bigg|\,
  (\bbm[\Lambda](\bbm[\mu]),\bbm[t])\in{\cal M},\;
  \Phi(\bbm[\Lambda](\bbm[\mu]),\bbm[t])< \frac{M}{2}\bigg\} .
\end{equation}
Since we do not know whether \eqref{utilde} is connected or not, so we will define $U$ as a
conveniently chosen connected component. Let $t_0\in(a,b)$ be fixed and choose $r_0>0$
sufficiently small such that
\begin{equation}\label{req1}
[t_0-4r_0, t_0+4r_0]\subset(a,b)
\end{equation}
and
\begin{equation}\label{req2}
\frac12h(t,t)+\frac12h(s,s)-g(t,s)\leq 0\;\;\forall t,s\in[t_0-4r_0,t_0+4r_0],\, t\neq s.
\end{equation}
Setting  $\bbm[\mu]_0=(1,1,1)$, $\bbm[t]_0=(t_0,t_0+r_0, t_0+2r_0, t_0+3r_0)$, then
$(\bbm[\Lambda](\bbm[\mu]_0), \bbm[t]_0)\in{\cal M}$ and, consequently,
$(\bbm[\mu]_0, \bbm[t]_0)$ belongs to \eqref{utilde} provided that $M$ is
sufficiently large. Now we are ready to define $U$, $K$ and $K_0$:
$$
U:=\mbox{ the connected component of }\eqref{utilde}\mbox{ containing }(\bbm[\mu]_0,\bbm[t]_0),
$$
$$
K=\left\{(\bbm[\Lambda](\bbm[\mu]),\bbm[t])\in{\cal M}:\ (\bbm[\mu],\bbm[t])\in\overline{U}
   \right\},
$$
$$
K_0=\left\{ (\bbm[\Lambda](\bbm[\mu]), \bbm[t])\in {\cal M}:\
     (\bbm[\mu], \bbm[t]) \in \partial U \right\}.
$$
Let us observe that, according to \eqref{coercivi}, the following inclusion holds:
\begin{equation}\label{inclu}
K_0\subset \left\{(\bbm[\Lambda],\bbm[t])\in K\,\bigg|\,
           \Phi(\bbm[\Lambda],\bbm[t])=\frac{M}{2}\right\}.
\end{equation}
$K$ is clearly isomorphic to $\overline{U}$ by the obvious isomorphism, and
$K_0 \approx \partial U$. In particular, $K$ and $K_0$ are compact sets and $K$ is connected.
Moreover we have $K_0\subset K\subset {\cal D}$.

Since $\Lambda_2=\Lambda_3$ by the definition of $K$, using \eqref{decre1} we obtain
\begin{equation}\label{kappa0}
-\sum_{i<j} (-i)^{i+j}\Lambda_i\Lambda_jg(t_i,t_j)
 \geq \Lambda_2\Lambda_3g(t_2,t_3)+\Lambda_1\Lambda_4g(t_1,t_4)\quad
 \forall (\bbm[\Lambda],\bbm[t])\in K.
\end{equation}
Roughly speaking, the configurations in $K$ have the crucial property that the negative
interaction terms associated to the couples of points with the same sign are dominated by the
positive interplay between the couples of points having opposite signs.

\subsection{An upper and a lower estimate for $\bbm[\tilde{\Psi}^*]$}
Let $\eta\in {\mathcal F}$, so $\eta:K\to {\mathcal D}$ is a continuous function such that
$\eta(\bbm[\Lambda],\bbm[t])=(\bbm[\Lambda],\bbm[t])$ for any $(\bbm[\Lambda],\bbm[t])\in K_0$.
Then we can compose the following maps
$$
(0,+\infty)^{3}\times (a,b)^4 \supset \overline{U} \longleftrightarrow K
 \stackrel{\eta}{\longrightarrow} \eta(K)
 \subset {\cal D}\stackrel{ {\cal H}}{\longrightarrow} (0,+\infty)^{3}\times(a,b)^4
$$
where
${\cal H}=({\cal H}_1,{\cal H}_2,\ldots,{\cal H}_7):{\cal D}\to (0,+\infty)^{3}\times(a,b)^4$
is defined by
$$
{\cal H}_1(\bbm[\Lambda],\bbm[t])=\Lambda_1\Lambda_2,\;\;\;
{\cal H}_2(\bbm[\Lambda],\bbm[t])=\Lambda_2\Lambda_3,\;\;\;
{\cal H}_3(\bbm[\Lambda],\bbm[t])=\Lambda_3\Lambda_4,
$$
$$
{\cal H}_4(\bbm[\Lambda],\bbm[t])=t_1,\;\;
{\cal H}_5(\bbm[\Lambda],\bbm[t])=t_2,\;\;
{\cal H}_6(\bbm[\Lambda],\bbm[t])=t_3,\;\;
{\cal H}_7(\bbm[\Lambda],\bbm[t])=t_4.
$$
We set
$$
T:\overline{U}\to  (0,+\infty)^{3}\times (a,b)^4
$$
the resulting composition. Clearly $T$ is a continuous map. We claim that $T=id$ on
$\partial U$. Indeed, if $(\bbm[\mu],\bbm[t])\in\partial U$, then by construction
$(\bbm[\Lambda](\bbm[\mu]), \bbm[t])\in K_0$; consequently
$\eta(\bbm[\Lambda](\bbm[\mu]), \bbm[t])=(\bbm[\Lambda](\bbm[\mu]), \bbm[t])$, by which,
using the definitions \eqref{firstform},
$$
{\cal H}_1(\bbm[\Lambda](\bbm[\mu]), \bbm[t])=\frac{\mu_1}{\sqrt{\mu}}\sqrt{\mu}=\mu_1,
$$
$$
{\cal H}_2(\bbm[\Lambda](\bbm[\mu]), \bbm[t])=\sqrt{\mu}\sqrt{\mu}=\mu
$$
$$
{\cal H}_3(\bbm[\Lambda](\bbm[\mu]), \bbm[t])=\sqrt{\mu}\frac{\mu_4}{\sqrt{\mu}}=\mu_4 .
$$
This proves that $T=id$ on $\partial U$. The theory of the topological degree assures that
$$
\deg(T, U,(\bbm[\mu]_0,\bbm[t]_0))
 = \deg (id, U, (\bbm[\mu]_0,\bbm[t]_0))=1.
$$
Then there exists $(\bbm[\mu]^\eta,\bbm[s]^\eta)\in U$ such that
$T(\bbm[\mu]^\eta,\bbm[s]^\eta)=(\bbm[\mu]_0,\bbm[t]_0)$, i.e., if we set
$(\bbm[\Lambda]^\eta, \bbm[t]^\eta):=\eta(\bbm[\Lambda](\bbm[\mu]^\eta), \bbm[s]^\eta)\in\eta(K)$,
\begin{equation}\label{chat3}
\Lambda_1^\eta\Lambda_2^\eta=\Lambda_2^\eta\Lambda_3^\eta=\Lambda_3^\eta\Lambda_4^\eta=1.
\end{equation}
\begin{equation}\label{chat4}
\bbm[t]^\eta=\bbm[t]_0.
\end{equation}
Using \eqref{req2}, and taking into account that
$\Lambda_1^\eta=\Lambda_3^\eta,$ $ \Lambda_2^\eta=\Lambda_4^\eta$ by \eqref{chat3}, we obtain
\begin{equation}\label{har1}
\frac12(\Lambda_1^\eta)^2h(t_1^0,t_1^0)+\frac12(\Lambda_3^\eta)^2h(t_3^0,t_3^0)
 -\Lambda_1^\eta\Lambda_3^\eta g(t_1^0,t_3^0)\leq 0,
\end{equation}
\begin{equation}\label{har2}
\frac12(\Lambda_2^\eta)^2h(t_2^0,t_2^0)+\frac12(\Lambda_4^\eta)^2h(t_4^0,t_4^0)
 -\Lambda_2^\eta\Lambda_4^\eta g(t_2^0,t_4^0)\leq 0.
\end{equation}
Furthermore by \eqref{chat3} we also deduce
\begin{equation}\label{har3}
\Lambda_1^\eta\Lambda_4^\eta=\frac{1}{\Lambda_2^\eta}\frac{1}{\Lambda_3^\eta}
 =\frac{1}{\Lambda_2^\eta\Lambda_3^\eta}=1,\quad
\Lambda_1^\eta\Lambda_2^\eta\Lambda_3^\eta\Lambda_4^\eta
 =(\Lambda_1^\eta\Lambda_2^\eta)(\Lambda_3^\eta\Lambda_4^\eta)=1.
\end{equation}
Combining \eqref{chat4}-\eqref{har1}-\eqref{har2}-\eqref{har3} with the definition of
$\tilde{\Psi}$ we get
\begin{equation*}
\tilde{\Psi}(\bbm[\Lambda]^\eta,\bbm[t]^\eta)
 \leq g(t_1^0, t_2^0)+g(t_2^0, t_3^0)+g(t_3^0, t_4^0)+g(t_1^0, t_4^0).
\end{equation*}
Then we can estimate
\begin{equation*}
\min_{(\hbox{\scriptsize$\bbm[\Lambda]$},\hbox{\scriptsize$\bbm[t]$})\in K}
 \tilde{\Psi}(\eta(\bbm[\Lambda],\bbm[t]))
\leq \tilde{\Psi}( \bbm[\Lambda]^\eta,\bbm[t]^\eta)
\leq g(t_1^0, t_2^0)+g(t_2^0, t_3^0)+g(t_3^0, t_4^0)+g(t_1^0, t_4^0).
\end{equation*}
By taking the supremum for all the maps $\eta\in {\cal F}$, we conclude
\begin{equation}\label{rox}
\tilde{\Psi}^*
 =\sup_{\eta\in {\cal F}}\min_{(\hbox{\scriptsize$\bbm[\Lambda]$},\hbox{\scriptsize$\bbm[t]$})\in K}
  \tilde{\Psi}(\eta(\bbm[\Lambda],\bbm[t]))
 \leq g(t_1^0, t_2^0)+g(t_2^0, t_3^0)+g(t_3^0, t_4^0)+g(t_1^0, t_4^0).
\end{equation}
On the other hand, by taking $\eta=id$ and using \eqref{coecoe} and \eqref{kappa0},
\begin{equation}\label{roxx}
\tilde{\Psi}^*
 \geq \min_{(\hbox{\scriptsize$\bbm[\Lambda]$},\hbox{\scriptsize$\bbm[t]$})\in K}
      \tilde{\Psi}(\bbm[\Lambda],\bbm[t])
 \geq -8\log^+\frac{2}{\sqrt{H_0}}.
\end{equation}

\subsection{Proof of (P2)}
Let us first recall that the upper estimate for $\tilde{\Psi}^*$ obtained in \eqref{rox} holds
for any $M$ sufficiently large. Then, by using \eqref{inclu}, the max-min inequality (P2) will
follow once we have proved that
\begin{equation}\label{rocco}
\min_{(\hbox{\scriptsize$\bbm[\Lambda]$},\hbox{\scriptsize$\bbm[t]$})\in K,\,
      \Phi(\hbox{\scriptsize$\bbm[\Lambda]$},\hbox{\scriptsize$\bbm[t]$})
=\frac{M}{2}}\tilde{\Psi}(\bbm[\Lambda],\bbm[t])
\to +\infty \quad \hbox{ as }M\to +\infty.
\end{equation}
To this aim, it will be convenient to provide a lower bound for the functional $\tilde{\Psi}$
over $K$. Combining \eqref{coecoe} and \eqref{kappa0}  we get
\begin{equation}\label{kappa}
\tilde{\Psi}(\bbm[\Lambda],\bbm[t])
 \geq\sum_{i=1}^4\frac{\Lambda_i^2}{4}h(t_i,t_i)+\sum_{i=1}^4|\log\Lambda_i|
      + \Lambda_2\Lambda_3g(t_2,t_3)+\Lambda_1\Lambda_4g(t_1,t_4) -8\log^+\frac{2}{\sqrt{H_0}}
\end{equation}
for any $(\bbm[\Lambda],\bbm[t])\in K$.

Now we are going to prove \eqref{rocco}. Indeed, let
$(\bbm[\Lambda]_n,\bbm[t]_n)
 =(\Lambda_1^n, \Lambda_2^n,\Lambda_3^n,\Lambda_4^n, t_1^n,t_2^n,t_3^n,t_4^n)\in K$
be such that
\begin{equation}\label{cover}
\Phi(\bbm[\Lambda]_n,\bbm[t]_n) \to +\infty.
\end{equation}

The definition of $\Phi$ implies that, up to a subsequence, the following four cases cover
all the possibilities for which \eqref{cover} may occur.
\begin{enumerate}
\item[{\bf(1)}] \textit{there exists ${\hat{\imath}}$ such that $\Lambda_{\hat{\imath}}^n\to 0$.}
\item[{\bf(2)}] \textit{there exists ${\hat{\imath}}$ such that
 $\Lambda_{\hat{\imath}}^n\to +\infty$.}
\item[{\bf(3)}] \textit{$t_1^n\to a$ or $t_4^n\to b$.}
\item[{\bf(4)}] \textit{for every $i$ the numbers $\Lambda_i^n$ are bounded from above and
 below by positive constants and there exist ${\hat{\imath}}< {\hat{\jmath}}$ such that
 $t^n_{\hat{\jmath}}-t^n_{\hat{\imath}}\to 0$.}
\end{enumerate}
\bigskip

If case (1), (2) or (3) holds, then by \eqref{kappa}, recalling \eqref{coercive}, we get
$\tilde{\Psi}(\bbm[\Lambda]_n,\bbm[t]_n)\to +\infty$, as required.

Assume that case (4) occurs. The definition of $\tilde{\Psi}$ combined with \eqref{kappa0}
implies  $$\tilde{\Psi}(\bbm[\Lambda]_n,\bbm[t]_n)\geq c\, g(t_2^n,t_3^n) -C$$ for suitable
positive constants $c, C$. Therefore, if ${\hat{\imath}}\le2$ and ${\hat{\jmath}}\ge3,$ we
get $t_3^n-t_2^n\to0$, hence $\tilde{\Psi}(\bbm[\Lambda]_n,\bbm[t]_n)\to +\infty$.

It remains to consider the case when, up to a subsequence
$$
t_3^n-t_2^n\geq a,\quad t^n_2-t^n_1\to 0,
$$
or
$$
t_3^n-t_2^n\geq a,\quad t^n_4-t^n_3\to 0
$$
for some $a>0$. Then we deduce $t_j^n-t_i^n\geq a$ for every  $i\leq 2<3\leq j$. Since the
Green's function $g$ is smooth on the compact sets disjoint from the diagonal, by the
definition of $\tilde{\Psi}$ we get
$$
\begin{aligned}
\tilde{\Psi}(\bbm[\Lambda]_n,\bbm[t]_n)\geq c'g(t_1^n,t_2^n)+c'g(t_3^n,t_4^n)-C'
\end{aligned}
$$
for some $c',\,C'>0$ and then we conclude
$$
\tilde{\Psi}(\bbm[\Lambda]_n,\bbm[t]_n)\to +\infty.
$$

\bigskip

\subsection{Proof of (P3)}
We shall prove that  (P3) holds provided that $M$ is sufficiently large. First we recall that
the upper and the lower estimates for $\Psi^*$ obtained in \eqref{rox} and \eqref{roxx}
holds for any $M$ sufficiently large. Then we proceed by contradiction: assume that there exist
$(\bbm[\Lambda]_{n}, \bbm[t]_n)
 =(\Lambda_{1}^n,\Lambda_2^n,\Lambda_3^n,\Lambda_4^n, t_1^n, t_2^n, t_3^n, t_4^n)
 \in\mathcal{M}$
and a vector $(\beta_1^n, \beta_2^n)\neq (0,0)$  such that:
$$ \Phi (\bbm[\Lambda]_{n}, \bbm[t]_n) = n, $$
$$ \tilde{\Psi}(\bbm[\Lambda]_{n}, \bbm[t]_n)=O(1), $$
$$ \beta_1^n \nabla \tilde{\Psi} (\bbm[\Lambda]_{n}, \bbm[t]_n)
    + \beta_2^n\nabla \Phi(\bbm[\Lambda]_{n}, \bbm[t]_n)=0.$$

The last expression means read as $\nabla \tilde{\Psi} (\bbm[\Lambda]_{n}, \bbm[t]_n) $ and
$\nabla \Phi(\bbm[\Lambda]_{n},\bbm[t]_n)$ are linearly dependent. Observe that, according to
the Lagrange Theorem, this contradicts the nondegeneracy of $\nabla \tilde{\Psi}$ on the
tangent space at the level $\Psi^*$.

Without loss of generality we may assume
\begin{equation}\label{winny}
(\beta_1^n)^2+(\beta_2^n)^2=1\hbox{ and } \beta_1^n+\beta_2^n\geq 0.
\end{equation}
Considering $ \Phi (\bbm[\Lambda]_{n}, \bbm[t]_n)+ \tilde{\Psi}(\bbm[\Lambda]_{n}, \bbm[t]_n)$
and $\Phi (\bbm[\Lambda]_{n}, \bbm[t]_n)- \tilde{\Psi}(\bbm[\Lambda]_{n}, \bbm[t]_n)$ we obtain,
respectively,
\begin{equation}\label{sist+}
\sum_{i=1}^4(\Lambda_i^n)^2 h(t_i^n,t_i^n)
 +2\sum_{i<j,\, (-1)^{i+j}=-1}\Lambda_i^n\Lambda_j^ng(t_i^n, t_j^n)
 -2\log (\Lambda_1^n\Lambda_2^n\Lambda_3^n\Lambda_4^n)
=n+O(1)
\end{equation}
and
\begin{equation}\label{sist-}
2\Lambda_1^n\Lambda_3^ng(t_1^n, t_3^n)+2\Lambda_2^n\Lambda_4^ng(t_2^n, t_4^n)=n+O(1).
\end{equation}
The identities
$\beta_1^n\frac{\partial\tilde{\Psi}}{\partial t_i}(\bbm[\Lambda]_{n}, \bbm[t]_n)
 + \beta_2^n\frac{\partial{\Phi}}{\partial t_i}(\bbm[\Lambda]_{n}, \bbm[t]_n)=0$
imply
\begin{equation}\label{acca}
(\beta_1^n+\beta_2^n)(\Lambda_1^n)^2\frac{\partial h}{\partial t}(t_i^n, t_i^n)
 -\sum_{j=1\atop j\neq i}^4((-1)^{i+j}\beta_1^n-\beta_2^n)\Lambda_i^n\Lambda_j^n
  \frac{\partial g}{\partial t}(t_i^n, t_j^n)
=0\quad \forall i=1,2,3,4.
\end{equation}
Moreover, from
$\beta_1^n\frac{\partial\tilde{\Psi}}{\partial \Lambda_i}(\bbm[\Lambda]_{n}, \bbm[t]_n)
 + \beta_2^n\frac{\partial{\Phi}}{\partial \Lambda_i}   (\bbm[\Lambda]_{n}, \bbm[t]_n) = 0$
we obtain the following four identities:
\begin{equation}\label{motiv}
(\beta_1^n+\beta_2^n)(\Lambda_i^n)^2 h(t_i^n,t_i^n)
 -\Lambda_i^n\sum_{j,j\neq i}((-1)^{i+j}\beta_1^n-\beta_2^n)\Lambda_j^ng(t_i^n, t_j^n)
=\beta_1^n+\beta_2^n\quad \forall i=1,2,3,4,
\end{equation}
by which, considering the sum in $i=1,2,3,4$,
\begin{equation}\label{sist3}
(\beta_1^n+\beta_2^n)\sum_{i=1}^4(\Lambda_i^n)^2 h(t_i^n,t_i^n)
 -2\sum_{i<j}((-1)^{i+j}\beta_1^n-\beta_2^n)\Lambda_i^n\Lambda_j^ng(t_i^n, t_j^n)
=4(\beta_1^n+\beta_2^n)
\end{equation}
which is equivalent to
\begin{equation}\label{sist3eq}
\beta_1^n\big(\tilde\Psi(\bbm[\Lambda]_{n}, \bbm[t]_n)
 +\log(\Lambda^{n}_1\Lambda_2^n\Lambda_3^n\Lambda_4^n)\big)
 +\beta_2^n\big(n+\log(\Lambda^{n}_1\Lambda_2^n\Lambda_3^n\Lambda_4^n)\big)
=2(\beta_1^n+\beta_2^n).
\end{equation}

Observe that by \eqref{sist+} we have
$\log(\Lambda^{n}_1\Lambda_2^n\Lambda_3^n\Lambda_4^n)\geq -\frac{n}{2}+O(1)$, while, by
\eqref{tele}, $(\Lambda_i^n)^2\leq \frac{4}{H_0}n+O(1)$ and hence
$\log(\Lambda^{n}_1\Lambda_2^n\Lambda_3^n\Lambda_4^n)\leq 2\log n+O(1)$. Then we easily obtain
$$
n+\log(\Lambda^{n}_1\Lambda_2^n\Lambda_3^n\Lambda_4^n)\to +\infty \quad\hbox{ and }\quad
\frac{\log(\Lambda^{n}_1\Lambda_2^n\Lambda_3^n\Lambda_4^n)}
     {n+\log(\Lambda^{n}_1\Lambda_2^n\Lambda_3^n\Lambda_4^n)}
\leq o(1).
$$
Multiplying \eqref{sist3eq} by $\beta_1^n$ we get
$$
\beta_1^n\beta_2^n
=2\beta_1^n\frac{\beta_1^n+\beta_2^n}{n+\log(\Lambda^{n}_1\Lambda_2^n\Lambda_3^n\Lambda_4^n)}
 -(\beta_1^n)^2\frac{O(1)+\log(\Lambda^{n}_1\Lambda_2^n\Lambda_3^n\Lambda_4^n)}
    {n+\log(\Lambda^{n}_1\Lambda_2^n\Lambda_3^n\Lambda_4^n)}
\geq o(1).
$$
Combining this with \eqref{winny} we have
\begin{equation}\label{pooh}
\beta_1^n\geq o(1),\quad \beta_2^n\geq o(1), \quad 2\geq\beta_1^n+\beta_2^n\geq 1+o(1).
\end{equation}
Using \eqref{pooh}, we can divide the identities \eqref{motiv} by $\beta_1^n+\beta_2^n$. Then
we obtain:
\begin{equation}\label{geppy1}
(\Lambda_1^n)^2h(t_1^n,t_1^n)+\Lambda_1^n \Lambda_2^ng(t_1^n, t_2^n)
 -\frac{\beta_1^n-\beta_2^n}{\beta_1^n+\beta_2^n}\Lambda_1^n\Lambda_3^ng(t_1^n, t_3^n)
 +\Lambda_1^n \Lambda_4^ng(t_1^n, t_4^n)
= 1,
\end{equation}
\begin{equation}\label{geppy2}
(\Lambda_2^n)^2h(t_2^n,t_2^n)+\Lambda_2^n\Lambda_1^n g(t_1^n, t_2^n)
 +\Lambda_2^n\Lambda_3^ng(t_2^n,t_3^n)
 -\frac{\beta_1^n-\beta_2^n}{\beta_1^n+\beta_2^n}\Lambda_2^n\Lambda_4^ng(t_2^n, t_4^n)
= 1,
\end{equation}
\begin{equation}\label{geppy3}
(\Lambda_3^n)^2h( t_3^n,t_3^n)
 -\frac{\beta_1^n-\beta_2^n}{\beta_1^n+\beta_2^n}\Lambda_3^n\Lambda_1^n g(t_1^n, t_3^n)
 +\Lambda_3^n\Lambda_2^ng(t_2^n, t_3^n)+\Lambda_3^n\Lambda_4^ng(t_3^n, t_4^n)
= 1,
\end{equation}
\begin{equation}\label{geppy4}
(\Lambda_4^n)^2h( t_4^n,t_4^n)+\Lambda_1^n\Lambda_4^n g(t_1^n, t_4^n)
 -\frac{\beta_1^n-\beta_2^n}{\beta_1^n+\beta_2^n}\Lambda_2^n\Lambda_4^ng(t_2^n, t_4^n)
 +\Lambda_3^n\Lambda_4^ng(t_3^n, t_4^n)
= 1.
\end{equation}

Up to a subsequence, we may assume
$$
t_i^n\to \bar t_i\in [a,b]\quad \forall i=1,2,3,4.
$$
In what follows  at many steps of the arguments we will pass to a subsequence, without
further notice. We will often use the symbol $c$ or $C$ for denoting different positive
constants independent on $n$. The value of $c$, $C$ is allowed to vary from line to line (and
also in the same formula). Motivated by \eqref{sist3}, we distinguish five cases which will
all lead to a contradiction.

\bigskip

{\bf Case 1. Avoiding blowing up of parameters I} \textit{Suppose the following holds:
\begin{equation}\label{molly}
(\beta_1^n-\beta_2^n)\Lambda_1^n\Lambda_3^ng(t_1^n, t_3^n)\to +\infty, \quad
(\beta_1^n-\beta_2^n)\Lambda_2^n\Lambda_4^ng(t_2^n, t_4^n)\to +\infty.
\end{equation}
}

\medskip

Then, in particular $\beta_1^n>\beta_2^n$ and, dividing \eqref{geppy1} by
$\frac{\beta_1^n-\beta_2^n}{\beta_1^n+\beta_2^n}\Lambda_1^n\Lambda_3^ng(t_1^n, t_3^n)$, we get

\begin{equation}\label{eli1}
\frac{\Lambda_2^n}{\Lambda_3^n}\cdot\frac{\beta_1^n+\beta_2^n}{\beta_1^n-\beta_2^n}
\leq \frac{\Lambda_2^n}{\Lambda_3^n}\cdot
     \frac{\beta_1^n+\beta_2^n}{\beta_1^n-\beta_2^n}\frac{g(t_1^n, t_2^n)}{g(t_1^n, t_3^n)}
\leq 1+o(1).
\end{equation}
where  the first inequality follows by \eqref{decre1}. Analogously, dividing \eqref{geppy4}
by $\frac{\beta_1^n-\beta_2^n}{\beta_1^n+\beta_2^n}\Lambda_2^n\Lambda_4^ng(t_2^n, t_4^n)$, and
using again \eqref{decre1}, we have
\begin{equation}\label{eli2}
\frac{\Lambda_3^n}{\Lambda_2^n}\cdot\frac{\beta_1^n+\beta_2^n}{\beta_1^n-\beta_2^n}
\leq \frac{\Lambda_3^n}{\Lambda_2^n}\cdot
     \frac{\beta_1^n+\beta_2^n}{\beta_1^n-\beta_2^n}\frac{g(t_3^n, t_4^n)}{g(t_2^n, t_4^n)}
\leq 1+o(1).
\end{equation}
\eqref{eli1} and \eqref{eli2} give
\begin{equation*}
\frac{\beta_1^n+\beta_2^n}{\beta_1^n-\beta_2^n}\leq 1+o(1)
\end{equation*}
which implies, using \eqref{pooh},
\begin{equation}\label{anto1}
\beta_2^n=o(1),\quad\beta_1^n=1+o(1).
\end{equation}
Inserting this into \eqref{eli1}-\eqref{eli2}  we achieve
\begin{equation}\label{anto2}
\Lambda_2^n=\Lambda_3^n(1+o(1)),
\end{equation}
and
\begin{equation}\label{anto3}
g(t_1^n, t_2^n)=g(t_1^n, t_3^n)(1+o(1)),\quad g(t_3^n, t_4^n)=g(t_2^n, t_4^n)(1+o(1)).
\end{equation}
Using \eqref{anto1}-\eqref{anto3} and \eqref{molly}, the equations
\eqref{geppy1}-\eqref{geppy4} lead to:
\begin{equation}\label{kat1}
(\Lambda_1^n)^2h( t_1^n,t_1^n)+\Lambda_1^n \Lambda_4^ng(t_1^n, t_4^n)
=o\big(\Lambda_1^n\Lambda_3^ng(t_1^n, t_3^n)\big),
\end{equation}
\begin{equation}\label{kat2}
(\Lambda_2^n)^2h( t_2^n,t_2^n)+\Lambda_2^n\Lambda_1^ng(t_1^n, t_2^n)
 +\Lambda_2^n\Lambda_3^ng(t_2^n, t_3^n)
=(1+o(1))\Lambda_2^n\Lambda_4^ng(t_2^n, t_4^n),
\end{equation}
\begin{equation}\label{kat3}
(\Lambda_3^n)^2h( t_3^n,t_3^n)+\Lambda_2^n\Lambda_3^ng(t_2^n, t_3^n)
 +\Lambda_3^n\Lambda_4^ng(t_3^n, t_4^n)
=(1+o(1))\Lambda_1^n\Lambda_3^ng(t_1^n, t_3^n),
\end{equation}
\begin{equation}\label{kat4}
(\Lambda_4^n)^2h( t_4^n,t_4^n)+\Lambda_1^n\Lambda_4^n g(t_1^n, t_4^n)
=o\big(\Lambda_2^n\Lambda_4^ng(t_2^n, t_4^n)\big).
\end{equation}
Combining \eqref{kat2}-\eqref{kat3} with \eqref{anto2}-\eqref{anto3} we obtain
$$
\begin{aligned}
\Lambda_2^n\Lambda_1^ng(t_1^n, t_2^n)
&\leq (1+o(1))\Lambda_2^n\Lambda_4^ng(t_2^n, t_4^n)=(1+o(1))\Lambda_3^n\Lambda_4^ng(t_3^n, t_4^n)\\
&\leq (1+o(1))\Lambda_1^n\Lambda_3^ng(t_1^n, t_3^n)= (1+o(1)) \Lambda_2^n\Lambda_1^ng(t_1^n, t_2^n).
\end{aligned}
$$
Then all the above inequalities are actually equalities, by which \eqref{kat2}-\eqref{kat3}
can be rewritten as
\begin{equation*}\label{kat22}
(\Lambda_2^n)^2h( t_2^n,t_2^n)+\Lambda_2^n\Lambda_3^ng(t_2^n, t_3^n)
=o\big(\Lambda_2^n\Lambda_4^ng(t_2^n, t_4^n)\big),
\end{equation*}
\begin{equation*}\label{kat33}
(\Lambda_3^n)^2h( t_3^n,t_3^n)+\Lambda_2^n\Lambda_3^ng(t_2^n, t_3^n)
=o\big(\Lambda_1^n\Lambda_3^ng(t_1^n, t_3^n)\big).
\end{equation*}
Now \eqref{decre2} applies and gives together with \eqref{anto2}
$$
o(\Lambda_4^n)=\Lambda_3^n=(1+o(1)) \Lambda_2^n=o(\Lambda_1^n).
$$
Substituting in \eqref{kat1} and \eqref{kat4} yields
\begin{equation}\label{flam1}
h( t_1^n,t_1^n),\;g(t_1^n, t_4^n)=o(g(t_1^n, t_3^n)), \quad
h(t_4^n,t_4^n),\,g(t_1^n, t_4^n)=o(g(t_2^n, t_4^n)).
\end{equation}
We will derive a contradiction from \eqref{anto3} and \eqref{flam1}. Indeed, by
$h( t_1^n,t_1^n)=o(g(t_1^n, t_3^n))$ we deduce  $g(t_1^n, t_3^n)\to +\infty$, hence
$|t_1^n-t_3^n|\to 0$. Analogously by $h(t_4^n,t_4^n)=o(g(t_2^n, t_4^n))$ we get
$|t_2^n-t_4^n|\to 0$. Therefore we are in the following situation
$$
t_1^n,\,t_2^n, \, t_3^n,\, t_4^n\to \bar t\in[a,b]\quad \forall i=1,2,3,4.
$$
Now, if $\bar t=a$, then Lemma \ref{robin} yields
$$
h(t_1^n, t_4^n)=\frac{1+o(1)}{\sigma_N(N-2)(t_4^n+t_1^n-2a)^{N-2}}
 \leq \frac{1+o(1)}{\sigma_N(N-2)(2t_1-2a)^{N-2}}=(1+o(1))h(t_1^n,t_1^n)
$$
and therefore, using \eqref{flam1},
$$
\frac{|t_1^n-t_3^n|^{N-2}}{|t_1^n-t_4^n|^{N-2}}
=\frac{g(t_1^n, t_4^n)+h(t_1^n, t_4^n)}{g(t_1^n, t_3^n)+h(t_1^n, t_3^n)}
\leq \frac{g(t_1^n, t_4^n)+h(t_1^n, t_4^n)}{g(t_1^n, t_3^n)}
=o(1)
$$
and then $t_3^n-t_1^n=o(t_4^n-t_1^n)$. On the other hand, using again Lemma \ref{robin},
$$
h(t_1^n, t_4^n) = \frac{1+o(1)}{\sigma_N(N-2)(t_4^n+t_1^n-2a)^{N-2}}
 \leq \frac{1+o(1)}{\sigma_N(N-2)(t_4-a)^{N-2}} = 2^{N-2}(1+o(1))h(t_4^n,t_4^n).
$$
Now \eqref{flam1} leads to
$$
\frac{|t_2^n-t_4^n|^{N-2}}{|t_1^n-t_4^n|^{N-2}}
= \frac{g(t_1^n, t_4^n)+h(t_1^n, t_4^n)}{h(t_2^n, t_4^n)+h(t_2^n, t_4^n)}
\leq \frac{g(t_1^n, t_4^n)+h(t_1^n, t_4^n)}{g(t_2^n, t_4^n)}
= o(1),
$$
hence $t_4^n-t_2^n=o(t_4^n-t_1^n)$. Combining this with $t_3^n-t_1^n=o(t_4^n-t_1^n)$ we obtain a
contradiction. An analogous argument applies to the case $\bar t=b$. Finally assume
$\bar t\in (a,b)$. Then $h(t_i^n, t_j^n)=O(1)$ for every $i,\,j$, therefore \eqref{flam1}
yields
$$
\frac{|t_i^n-t_j^n|^{N-2}}{|t_1^n-t_4^n|^{N-2}}
= \frac{g(t_1^n, t_4^n)+O(1)}{g(t_i^n, t_j^n)+O(1)}=o(1)\quad
\hbox{ for } (i,j)=(1,3), \,(2,4).
$$
This gives $t_3^n-t_1^n=o(t_4^n-t_1^n)$ and  $t_4^n-t_2^n=o(t_4^n-t_1^n)$ respectively, and the
contradiction arises as above.

\bigskip

{\bf Case 2: Avoiding blowing up of parameters II.}
\textit{Suppose the following holds:
\begin{equation}\label{molla}
(\beta_1^n-\beta_2^n)\Lambda_1^n\Lambda_3^ng(t_1^n, t_3^n)\to +\infty, \quad
 (\beta_1^n-\beta_2^n)\Lambda_2^n\Lambda_4^ng(t_2^n, t_4^n) \leq C.
\end{equation}
The analogous holds by interchanging the roles of the couples of indexes $(1,3)$ and $(2,4)$.}

\bigskip

Then in particular there holds $\beta_1^n>\beta_2^n$. Using \eqref{geppy2}, \eqref{geppy4} and
the second inequality in \eqref{molla} we obtain
\begin{equation}\label{gin11}
(\Lambda_2^n)^2h(t_2^n, t_2^n),\,\Lambda_1^n \Lambda_2^ng(t_1^n, t_2^n),\,
 \Lambda_2^n\Lambda_3^ng(t_2^n, t_3^n)
\leq C,
\end{equation}
\begin{equation}\label{gin44}
(\Lambda_4^n)^2h(t_4^n, t_4^n),\,\Lambda_1^n\Lambda_4^n g(t_1^n, t_4^n),\,
 \Lambda_3^n\Lambda_4^ng(t_3^n, t_4^n)
\leq C.
\end{equation}
By inserting \eqref{gin11}-\eqref{gin44} into \eqref{geppy1} and \eqref{geppy3}, we obtain
\begin{equation}\label{gola}
(\Lambda_1^n)^2h(t_1^n, t_1^n),\,(\Lambda_3^n)^2h(t_3^n, t_3^n)
= \frac{\beta_1^n-\beta_2^n}{\beta_1^n+\beta_2^n}\Lambda_1^n\Lambda_3^ng(t_1^n, t_3^n)+O(1)
\to +\infty.
\end{equation}
We distinguish three cases. First assume that there exists $i_0\in\{1,2,3,4\}$ such that
\begin{equation}\label{firstcase}
t_i^n\to a\quad \forall 1\leq i\leq i_0,\qquad  |t_i^n-a|\geq c \quad\forall i>i_0.
\end{equation}
By adding \eqref{acca} for $i=1,\ldots,i_0$ we obtain
\begin{equation}\label{acce1}
(\beta_1^n+\beta_2^n)\sum_{i=1}^{i_0}(\Lambda_i^n)^2\frac{\partial h}{\partial t}(t_i^n, t_i^n)
 -\sum_{i=1}^{i_0}\sum_{j=1\atop j\neq i}^4((-1)^{i+j}\beta_1^n-\beta_2^n)\Lambda_i^n\Lambda_j^n
  \frac{\partial g}{\partial t}(t_i^n, t_j^n)
= 0.
\end{equation}
Now $|t_i^n-t_j^n|\geq c$ for $i\leq i_0$ and $j>i_0$ imply
\begin{equation}\label{go}\frac{\partial g}{\partial t}(t_i^n, t_j^n)=O(1)\quad
 \forall i\leq i_0,\forall j>i_0.
\end{equation}
Considering the sum for $i,j\leq i_0$ we observe that by Lemma \ref{robin}
$$
\frac{\partial g}{\partial t}(t_i^n, t_j^n)+\frac{\partial g}{\partial t}(t_j^n, t_i^n)
= -\frac{\partial h}{\partial t}(t_i^n, t_j^n)-\frac{\partial h}{\partial t}(t_j^n, t_i^n)
=\frac{2+o(1)}{\sigma_N(t_i^n+t_j^n-2a)^{N-1}}\quad \forall i,j\leq i_0,\, i\neq j.
$$
Therefore, using again Lemma \ref{robin}, the identity of \eqref{acce1} becomes
\begin{equation}\label{accaacca1}
\sum_{i=1}^{i_0}\frac{(\Lambda_i^n)^2(1+o(1))}{(2t_i^n-2a)^{N-1}}
 +2\sum_{i,j=1\atop i<j}^{i_0}\frac{(-1)^{i+j}\beta_1^n-\beta_2^n}{\beta_1^n+\beta_2^n}\cdot
  \frac{\Lambda_i^n\Lambda_j^n(1+o(1))}{(t_j^n+t_i^n-2a)^{N-1}}
= \sum_{i\leq i_0<j}O( \Lambda_i^n\Lambda_j^n).
\end{equation}
In order to estimate the last sum, we will prove that
\begin{equation}\label{anx}
\Lambda_i^n\Lambda_j^n=o\bigg(\frac{(\Lambda_i^n)^2}{(2t_i^n-2a)^{N-1}}\bigg)+O(1)\quad
 \forall i\leq i_0<j.
\end{equation}
Indeed, if $i\leq i_0<j$ and $(i,j)\neq (1,3)$, then, either $j=2$ or $j=4$, and, as a
consequence of \eqref{gin11}-\eqref{gin44}, $\Lambda_2^n,\, \Lambda_4^n=O(1)$; therefore
$\Lambda_i^n\Lambda_j^n \leq \frac12(\Lambda_i^n)^2+\frac12 (\Lambda_j^n)^2
 \leq \frac12(\Lambda_i^n)^2+C$ and \eqref{anx} holds true.
On the other hand, using \eqref{gola},
$$
\Lambda_1^n\Lambda_3^n
=(1+o(1))(\Lambda_1^n)^2\bigg(\frac{h(t_1^n, t_1^n)}{h(t_3^n, t_3^n)}\bigg)^{1/2}
\leq (1+o(1))(\Lambda_1^n)^2\bigg(\frac{h(t_1^n, t_1^n)}{H_0}\bigg)^{1/2}
$$
and \eqref{anx} follows by using Lemma \ref{robin}.

Next, in order to estimate the second sum in \eqref{accaacca1}, we claim that
\begin{equation}\label{anxanx}
\frac{\Lambda_i^n\Lambda_j^n}{(t_j^n+t_i^n-2a)^{N-1}}
=o\bigg(\frac{(\Lambda_i^n)^2}{(2t_i^n-2a)^{N-1}}+\frac{(\Lambda_j^n)^2}{(2t_j^n-2a)^{N-1}}\bigg)\
\quad \hbox{ if }i,j\leq i_0, (-1)^{i+j}=-1 .
\end{equation}
Indeed, take, for instance, the couple $(i,j)=(1,2)$; the other cases are analogous. The
claim is obvious if $\Lambda_2^n=o(\Lambda_1^n)$ or $\Lambda_1^n=o(\Lambda_2^n).$ Otherwise
$c\leq \frac{\Lambda_2^n}{\Lambda_1^n}\leq C$ and then, using \eqref{gin11} and \eqref{gola},
$\frac{h(t_2^n, t_2^n)}{h(t_1^n, t_1^n)}=o(1),$ by which, applying Lemma \ref{robin},
$t_1^n-a=o(t_2^n-a) $. This in turn implies $t_1^n-a=o(t_1^n+t_2^n-2a),$ and \eqref{anxanx}
follows.

Therefore, recalling that $\beta_1^n>\beta_2^n$, \eqref{accaacca1} becomes
\begin{equation*}
\sum_{i=1}^{i_0}\frac{(\Lambda_i^n)^2(1+o(1))}{(2t_i^n-2a)^{N-1}}\leq C.
\end{equation*}
Taking into account that
$\frac{(\Lambda_1^n)^2}{(2t_1^n-2a)^{N-1}}\geq c\frac{(\Lambda_1^n)^2h(t_1^n, t_1^n)}{2t_1^n-2a}
 \to +\infty$ by Lemma \ref{robin} and \eqref{gola}, the contradiction follows.

An analogous argument can be applied when there exists $i_0\in\{1,2,3,4\}$ such that
\begin{equation}\label{secondcase}
t_i^n\to b\quad \forall i_0\leq i\leq 4,\qquad  |t_i^n-b|\geq c \quad\forall i<i_0.
\end{equation}
So we may assume
\begin{equation}\label{thirdcase}
t_i^n\to\bar t_i \in (a,b)\quad \forall i=1,2,3,4.
\end{equation}
According to the assumption \eqref{prop1} we have either
$\frac{\partial h}{\partial t}(t_1^n, t_1^n)\leq0$ or
$\frac{\partial h}{\partial t}(t_3^n, t_3^n)\geq0$. Assume, for instance,
$$
\frac{\partial h}{\partial t}(t_1^n, t_1^n)\leq0
$$
(the case $\frac{\partial h}{\partial t}(t_3^n, t_3^n)\geq0$ can be treated analogously).
We set $\{1,2,3,4\}=I\cup J$ where
$$
I=\{i\,:\, |t_i^n-t_1^n|=o(|t_1^n-t_3^n|)\},\quad J=\{i\,;\, |t_i^n-t_1^n|\geq c(|t_1^n-t_3^n|)\}.
$$
It is obvious that $I=\{1\}$ or $I=\{1,2\}$. Then, adding \eqref{acca} for $i\in I$ we get
\begin{equation}\label{aggot}
\sum_{i\in I}\sum_{j=1\atop j\neq i}^4((-1)^{i+j}\beta_1^n-\beta_2^n)\Lambda_i^n\Lambda_j^n
 \frac{\partial g}{\partial t}(t_i^n, t_j^n)
\leq C(\Lambda_2^n)^2.
\end{equation}
Observe that
$$
\frac{\partial g}{\partial t}(t_1^n, t_2^n)+\frac{\partial g}{\partial t}(t_2^n, t_1^n)
= -\frac{\partial h}{\partial t}(t_1^n, t_2^n)-\frac{\partial h}{\partial t}(t_2^n, t_1^n)
= O(1)
$$
and $\Lambda_2^n\leq C$, $\Lambda_1^n\Lambda_2^n\leq C$, by \eqref{gin11}; therefore
\eqref{aggot} becomes
\begin{equation}\label{aggoss}
\sum_{i\in I}\sum_{ j\in J}((-1)^{i+j}\beta_1^n-\beta_2^n)
 \Lambda_i^n\Lambda_j^n\frac{\partial g}{\partial t}(t_i^n, t_j^n)
\leq C.
\end{equation}
According to the assumption \eqref{prop2} we have
$\frac{\partial g}{\partial t}(t, s)>0$ if $t<s$. Since all the sequences $t_i^n$ lie in a
compact subset of $\Omega$, Lemma \ref{robin} implies
\begin{equation}\label{forse}
c\,\frac{g(t_i^n, t_j^n)}{|t_{i}^n-t_{j}^n|}
\leq  \frac{\partial g}{\partial t}(t_i^n, t_j^n)
\leq C\frac{g(t_i^n, t_j^n)}{|t_{i}^n-t_{j}^n|} \quad \forall i,j=1,2,3,4, \;i< j.
\end{equation}
On the other hand, if $i\in I$ and $j\in J$, then  $i<j$ and $|t_i^n-t_j^n|\geq c|t_3^n-t_1^n|$
by the definition of $I, J$; therefore combining  \eqref{aggoss} and \eqref{forse}
we arrive at
\begin{equation}\label{aggo}
(\beta_1^n-\beta_2^n)\Lambda_1^n\Lambda_3^ng(t_1^n, t_3^n)
\leq C\sum_{(i,j)\in I\times J\atop (i,j)\neq (1,3)}|(-1)^{i+j}\beta_1^n-\beta_2^n|
 \Lambda_i^n\Lambda_j^ng(t_i^n, t_j^n)+C .
\end{equation}
This contradicts \eqref{molla}-\eqref{gin11}-\eqref{gin44}.

\bigskip

{\bf Case 3: Avoiding the boundary.} \textit{Suppose the following holds:
$|\beta_1^n-\beta_2^n|\Lambda_1^n\Lambda_3^ng(t_1^n, t_3^n)=O(1),\;
|\beta^n_1-\beta_2^n|\Lambda_2^n\Lambda_4^ng(t_2^n, t_4^n)=O(1)$,
$I_a:=\{i=1,2,3,4\,|\, \bar t_i=a\}\neq \emptyset $ and
\begin{equation}\label{izero}
(\beta_1^n+\beta_2^n)\sum_{(i,j)\in I_a}(\Lambda_{i}^n)^2 h(t_{i}^n, t_{i}^n)
 +2\sum_{i,j\in I_a \atop i<j }|(-1)^{i+j}\beta_1^n-\beta_2^n|\Lambda_{i}^n\Lambda_{j}^ng(t_{i}^n,t_{j}^n)
\geq c.
\end{equation}
Replacing $I_a$ with $I_b$ can be treated analogously. }

\bigskip

First of all we observe that \eqref{sist-} implies
\begin{equation}\label{pooh1}
\beta_1^n-\beta_2^n=o(1).
\end{equation}
Recalling that $(\beta_1^n)^2+(\beta_2^n)^2=1$ it follows that
\begin{equation}\label{pooh2}
\beta_1^n+\beta_2^n=\sqrt 2+o(1).
\end{equation}
Using \eqref{geppy1}-\eqref{geppy4} we obtain
\begin{equation}\label{bou1}
(\Lambda_i^n)^2h(t_i^n, t_i^n)\leq C\quad \forall i=1,2,3,4,
\end{equation}
hence $\Lambda_i^n\leq C$ for all $i=1,2,3,4$, and
\begin{equation}\label{bou}
\Lambda_1^n\Lambda_2^n g(t_1^n, t_2^n),\; \Lambda_1^n\Lambda_4^n g(t_1^n, t_4^n), \;
\Lambda_2^n\Lambda_3^n g(t_2^n, t_3^n),\; \Lambda_3^n\Lambda_4^n g(t_3^n, t_4^n)\leq C.
\end{equation}
Now we  multiply \eqref{acca} by $t_i^n-a$ and add for $i\in I_a$
\begin{equation}\label{acca11}
\sum_{i\in I_a}(\Lambda_i^n)^2\frac{\partial h}{\partial t}(t_i^n, t_i^n)(t_i^n-a)
 - \sum_{i\in I_a}\sum_{j=1\atop j\neq i}^4\frac{(-1)^{i+j}\beta_1^n-\beta_2^n}{\beta_1^n+\beta_2^n}
   \Lambda_i^n\Lambda_j^n\frac{\partial g}{\partial t}(t_i^n, t_j^n)(t_i^n-a)
= 0.
\end{equation}
We  estimate the terms in each  sum in order to obtain a contradiction. Lemma \ref{robin}
implies
\begin{equation*}
\frac{\partial h}{\partial t}(t_i^n, t_i^n)
=-\frac{1+o(1)}{\sigma_N(2t_i^n-2a)^{N-1}}
=-(N-2)(1+o(1))\frac{h(t_i^n, t_i^n)}{2(t_i^n-a)},
\quad\forall i\in I_a.
\end{equation*}
By the definition of $I_a$, there holds $|t_i^n-t_j^n|\geq c$ for $i\in I_a$ and $j\not\in I_a$.
This implies
\begin{equation}\label{goe}
\frac{\partial g}{\partial t}(t_i^n, t_j^n)=O(1)\quad
\forall i\in I_a,\,\forall j\not\in I_a.
\end{equation}
We split the second sum in \eqref{acca11} in two terms: those with $j\in I_a$ and those with
$j\not\in I_a$. We use  again Lemma \ref{robin} and, considering the sum for $i,j\in I_a$, we
observe that
$$
\begin{aligned}\
\frac{\partial g}{\partial t}(t_i^n, t_j^n)(t_i^n-a)
& +\frac{\partial g}{\partial t}(t_j^n, t_i^n)(t_j^n-a)
 =-\frac{1}{\sigma_N|t_j^n-t_i^n|^{N-2}}+\frac{1+o(1)}{\sigma_N(t_j^n-2a+t_i^n)^{N-2}}\\
& =-(N-2)g(t_i^n,t_j^n)+\frac{o(1)}{\sigma_N(t_j^n-2a+t_i^n)^{N-2}} \qquad
 \forall i,j\in I_a,\, i\neq j.
\end{aligned}
$$
On the other hand, it is straightforward to prove that the function
$\frac{\exp y}{|t-a|^{N-2}}$ is convex  for $t\geq a$, $ y\in \R$. Therefore
$$
\begin{aligned}
\frac{\Lambda_i^n\Lambda_j^n}{(t_j^n-2a+t_i^n)^{N-2}}
&= \frac{\exp\big(\frac{\log (\Lambda_i^n)^2}{2}+\frac{\log (\Lambda_j^n)^2}{2}\big)}
        {2^{N-2}(\frac{t_i^n+t_j^n}{2}-a)^{N-2}}
 \leq \frac{(\Lambda_i^n)^2}{2(2t_i^n-2a)^{N-2}}+\frac{(\Lambda_j^n)^2}{2(2t_j^n-2a)^{N-2}}\\
&\leq C((\Lambda_i^n)^2h(t_i^n, t_i^n)+(\Lambda_j^n)^2h(t_j^n, t_j^n))
 \leq C
 \quad \forall i,j\in I_a,\, i\neq j.
\end{aligned}
$$
Therefore \eqref{acca11} becomes
\begin{equation}\label{accaacca11}
(\beta_1^n+\beta_2^n)\sum_{i\in I_a}(\Lambda_i^n)^2h(t_i^n,t_i^n)
 -2\sum_{(i,j)\in I_a\atop i<j}((-1)^{i+j}\beta_1^n-\beta_2^n)\Lambda_i^n\Lambda_j^ng(t_i^n,t_j^n)
= o(1).
\end{equation}
If $I_a=\{1\}$ or $I_a=\{1,2\},$ then the left hand sides of \eqref{izero} and
\eqref{accaacca11} coincide, in contradiction with the right hand sides. If $I_a=\{1,2,3,4\}$,
then the contradiction arises by comparing \eqref{accaacca11} with \eqref{sist3} because of
\eqref{pooh2}. So it remains to consider the case $I_a=\{1,2,3\}$. We sum the identities
\eqref{motiv} for $i=1,2,3$ and subtract \eqref{accaacca11} and we obtain
$$
(\beta_1^n+\beta_2^n)\Lambda_1^n\Lambda_4^ng(t_1^n,t_4^n)
 -(\beta_1^n-\beta_2^n)\Lambda_2^n \Lambda_4^ng(t_2^n,t_4^n)
 +(\beta_1^n+\beta_2^n)\Lambda_3^n\Lambda_4^ng(t_3^n,t_4^n)
= 3(\beta_1^n+\beta_2^n)+o(1).
$$
However, combining  this with  \eqref{motiv} for $i=4$ gives
$$
(\beta_1^n+\beta_2^n)(\Lambda_4^n)^2h(t_4^n, t_4^n)+2(\beta_1^n+\beta_2^n)=o(1)
$$
and the contradiction arises because of \eqref{pooh2}.

\medskip

{\bf Case 4: Avoiding collisions.} \textit{Suppose the following holds:
$|\beta_1^n-\beta_2^n|\Lambda_1^n\Lambda_3^ng(t_1^n, t_3^n)=O(1),\;
|\beta^n_1-\beta_2^n|\Lambda_2^n\Lambda_4^ng(t_2^n, t_4^n)=O(1)$
and there exists  $i_0\neq j_0$ such that $\bar t_{i_0}=\bar t_{j_0}\in (a,b)$ and
$$
|(-1)^{i_0+j_0}\beta_1^n-\beta_2^n|\Lambda_{i_0}^n\Lambda_{j_0}^ng(t_{i_0}^n, t_{j_0}^n)\geq c .
$$ }

\bigskip

As in the previous case we immediately get \eqref{pooh1}--\eqref{bou}. Hence, in particular,
$\Lambda_i^n\leq C$ for any $i=1,2,3,4$. Set $\bar t=\bar t_{i_0}=\bar t_{j_0}\in (a,b)$ and
$I=\{i=1,2,3,4\,|\, \bar t_i=\bar t\}$. We split $I=I_1\cup I_2$ where
$$
I_1=\left\{i\in I\,\bigg|\,\exists j\in I,\, j\neq i\hbox{ s.t. }\;
  |(-1)^{i+j}\beta_1^n-\beta_2^n|\frac{\Lambda_{i}^n\Lambda_{j}^n}{|t_i^n-t_j^n|^{N-1}}
   \to +\infty \right\},
$$
and
$$
I_2=\left\{i\in I\,\bigg|\,\forall j\in I, \, j\neq i:\;
  |(-1)^{i+j}\beta_1^n-\beta_2^n|\frac{\Lambda_{i}^n\Lambda_{j}^n}{|t_i^n-t_j^n|^{N-1}}\leq C\right\}.
$$
Since the sequences $t_i^n$ lie in a compact subset of $\Omega$ for any $i\in I$, Lemma
\ref{robin} implies
$$
\frac{\partial g}{\partial t}(t_i^n, t_j^n)
=-\frac{t_{i}^n-t_{j}^n}{\sigma_N|t_{i}^n-t_{j}^n|^{N}}+O(1)\quad
 \forall i\in I,\,\forall j=1,2,3,4, \;i\neq j.
$$
Moreover, observe that
$\frac{1}{|t_{i_0}^n-t_{j_0}^n|^{N-1}}
 \geq \sigma_N(N-2)\frac{g(t_{i_0}^n, t_{j_0}^n)}{|t_{i_0}^n- t_{j_0}^n|}.$
Therefore, according to the assumptions,  $i_0, j_0\in I_1$. For any $i\in I_1$ we consider
\eqref{acca} and obtain
\begin{equation}\label{tres}
\sum_{j\in I_1, \,j\neq i}((-1)^{i+j}\beta_1^n-\beta_2^n)\Lambda_{i}^n\Lambda_{j}^n
 \frac{t_{i}^n-t_{j}^n}{|t_{i}^n-t_{j}^n|^{N}}
= O(1) \quad \forall i\in I_1.
\end{equation}
Using \eqref{tres} for $i_0$, we immediately get the existence of a third index
$j\in I_1$, $j\neq i_0, j_0$. Therefore $I_1$ has actually at least three elements.
Assume $I_1= \{1,2,3,4\}$. We look at \eqref{tres} for $i = 1$:
$$
(\beta_1^n+\beta_2^n)\frac{\Lambda_{1}^n\Lambda_{2}^n}{|t_{1}^n-t_{2}^n|^{N-1}}
 +(\beta_1^n+\beta_2^n)\frac{\Lambda_{1}^n\Lambda_{4}^n}{|t_{1}^n-t_{4}^n|^{N-1}}
= (\beta_1^n-\beta_2^n)\frac{\Lambda_{1}^n\Lambda_{3}^n}{|t_{1}^n-t_{3}^n|^{N-1}}+O(1)
\to +\infty
$$
which yields $\beta_1^n>\beta_2^n$. Dividing the identity by
$(\beta_1^n-\beta_2^n)\frac{\Lambda_{1}^n\Lambda_{3}^n}{|t_{1}^n-t_{3}^n|^{N-1}}$,
and using $|t_{1}^n-t_{2}^n|<|t_{1}^n-t_{3}^n|$, we get
$$
\frac{\beta_1^n+\beta_2^n}{\beta_1^n-\beta_2^n}\leq \frac{\Lambda_3^n}{\Lambda_2^n}(1+o(1)).
$$
Next we consider \eqref{tres} for $i = 4$ and proceed analogously, using now that
$|t_{3}^n-t_{4}^n|<|t_{2}^n-t_{4}^n|$. This leads to:
$\frac{\beta_1^n+\beta_2^n}{\beta_1^n-\beta_2^n}\leq \frac{\Lambda_2^n}{\Lambda_3^n}(1+o(1))$,
and so
$$
\frac{\beta_1^n+\beta_2^n}{\beta_1^n-\beta_2^n}\leq 1+o(1)
$$
in contradiction with \eqref{pooh1}-\eqref{pooh2}.

It remains to consider the case when $I_1$ has exactly three elements. If $I_1=\{1,2,4\}$,
then \eqref{tres} for $i=2$ gives
$$
(\beta_1^n+\beta_2^n)\frac{\Lambda_{1}^n\Lambda_{2}^n}{|t_{1}^n-t_{2}^n|^{N-1}}
= -(\beta_1^n-\beta_2^n)\frac{\Lambda_{2}^n\Lambda_{4}^n}{|t_{1}^n-t_{4}^n|^{N-1}}+O(1)
\to +\infty,
$$
which is absurd if $\beta_1^n\geq \beta_2^n$. On the other hand, by \eqref{tres} for  $i=4$
$$
(\beta_1^n+\beta_2^n)\frac{\Lambda_{1}^n\Lambda_{4}^n}{|t_{1}^n-t_{4}^n|^{N-1}}
= (\beta_1^n-\beta_2^n)\frac{\Lambda_{2}^n\Lambda_{4}^n}{|t_{2}^n-t_{4}^n|^{N-1}}+O(1)
\to +\infty
$$
which gives the contradiction in the case  $\beta_1^n<\beta_2^n$. An analogous argument applies
to the case $I_1=\{1,3,4\}$.

It remains to consider the cases $I_1=\{1,2,3\}$ and  $I_1=\{2,3,4\}.$ Assume, for instance,
$I_1=\{1,2,3\}$, the other case is similar. Then by \eqref{tres} we obtain
\begin{equation}\label{idde}
\frac{\Lambda_1^n\Lambda_2^n}{|t_2^n-t_1^n|^{N-1}}
= \frac{\beta_1^n-\beta_2^n}{\beta_1^n+\beta_2^n}\cdot
  \frac{\Lambda_1^n\Lambda_3^n}{|t_3^n-t_1^n|^{N-1}}+O(1)
=\frac{\Lambda_2^n\Lambda_3^n}{|t_3^n-t_2^n|^{N-1}}+O(1)
\to +\infty.
\end{equation}
In particular we have $\beta_1^n>\beta_2^n$. Using \eqref{pooh1}-\eqref{pooh2}, the first and
the second equality in \eqref{idde} give
$$
\Lambda_2^n=o(\Lambda_3^n),\quad \Lambda_2^n=o(\Lambda_1^n),
$$
respectively. Now we multiply the first identity in \eqref{idde} by $t_2^n-t_1^n$ and the
second by $t_3^n-t_2^n$ and, summing up, we obtain
 \begin{equation}\label{cassa}
(\beta_1^n+\beta_2^n)\frac{\Lambda_1^n\Lambda_2^n}{|t_2^n-t_1^n|^{N-2}}
 -(\beta_1^n-\beta_2^n)\frac{\Lambda_1^n\Lambda_3^n}{|t_3^n-t_1^n|^{N-2}}
 +(\beta_1^n+\beta_2^n)\frac{\Lambda_2^n\Lambda_3^n}{|t_3^n-t_2^n|^{N-2}}
= o(1).
\end{equation}
 We may also assume
\begin{equation}\label{elle4}
\Lambda_4^n\geq c.
\end{equation}
 Otherwise, if $\Lambda_4^n\to 0$, then \eqref{motiv} for $i=4$ would give
$$
(\Lambda_4^n)^2 h(t_4^n, t_4^n)+\Lambda_4^n(\Lambda_1^ng(t_1^n,t_4^n)+\Lambda_3^ng(t_3^n,t_4^n))
= 1+\frac{\beta_1^n-\beta_2^n}{\beta_1^n+\beta_2^n}\Lambda_2^n \Lambda_4^ng(t_2^n,t_4^n)
\geq 1
$$
by which either $(\Lambda_4^n)^2 h(t_4^n, t_4^n)\geq \frac12$ or
$\Lambda_4^n(\Lambda_1^ng(t_1^n,t_4^n)+\Lambda_3^ng(t_3^n,t_4^n))\geq \frac12$. If
$(\Lambda_4^n)^2 h(t_4^n, t_4^n)\geq \frac12$, then $h(t_4^n, t_4^n)\to +\infty$, and,
consequently, $t_4^n\to b$, so that we are again in the case 3. Otherwise, if
$\Lambda_4^n(\Lambda_1^ng(t_1^n,t_4^n)+\Lambda_3^ng(t_3^n,t_4^n))\geq \frac12$, then
$g(t_1^n,t_4^n)+g(t_3^n,t_4^n)\to +\infty$. So $t_4^n\to \bar t$ and then
$$
\frac{\Lambda_{1}^n\Lambda_{4}^n}{|t_1^n-t_4^n|^{N-1}}
 +\frac{\Lambda_{3}^n\Lambda_{4}^n}{|t_3^n-t_4^n|^{N-1}}
\geq \sigma_N(N-2)\Lambda_{1}^n\Lambda_{4}^n\frac{g(t_1^n, t_4^n)}{|t_1^n-t_4^n|}
 +\sigma_N(N-2)\Lambda_{3}^n\Lambda_{4}^n\frac{g(t_3^n, t_4^n)}{|t_3^n-t_4^n|}
\to +\infty,
$$
contradicting that $4\not \in I_1$.

Now we distinguish three cases. First assume
\begin{equation}\label{casse1}
\Lambda_1^n,\;\Lambda_2^n,\; \Lambda_3^n\to 0.
\end{equation}
Then \eqref{cassa} can be rewritten as
$$
(\beta_1^n+\beta_2^n)\Lambda_1^n\Lambda_2^ng(t_1^n,t_2^n)
 -(\beta_1^n-\beta_2^n)\Lambda_1^n \Lambda_3^ng(t_1^n,t_3^n)
 +(\beta_1^n+\beta_2^n)\Lambda_2^n\Lambda_3^ng(t_2^n,t_3^n)
= o(1).
$$
We sum the identities \eqref{motiv} in $i=1,2,3$ and, using the above estimate and
\eqref{casse1}, we obtain
$$
(\beta_1^n+\beta_2^n)\Lambda_1^n\Lambda_4^ng(t_1^n,t_4^n)
 -(\beta_1^n-\beta_2^n)\Lambda_2^n \Lambda_4^ng(t_2^n,t_4^n)
 +(\beta_1^n+\beta_2^n)\Lambda_3^n\Lambda_4^ng(t_3^n,t_4^n)
= 3(\beta_1^n+\beta_2^n)+o(1).
$$
However, combining this with  \eqref{motiv} for $i=4$ gives
$$
(\beta_1^n+\beta_2^n)(\Lambda_4^n)^2h(t_4^n, t_4^n)+2(\beta_1^n+\beta_2^n)=o(1)
$$
and a contradiction arises because of \eqref{pooh2}.

Now assume that
\begin{equation}\label{casse2}
\Lambda_1^n,\,\Lambda_2^n\to 0,\quad \Lambda_3^n\geq c.
\end{equation}
Then $\Lambda_1^n\Lambda_2^n=o(\Lambda_2^n\Lambda_3^n)$.  According to \eqref{idde} we have
$\frac{\Lambda_1^n\Lambda_2^n}{|t_2^n-t_1^n|^{N-1}}
 =(1+o(1))\frac{\Lambda_2^n\Lambda_3^n}{|t_3^n-t_2^n|^{N-1}}$,
from which we deduce $t_2^n-t_1^n=o(t_3^n-t_2^n)$.
Consequently
$\frac{\Lambda_1^n\Lambda_2^n}{|t_2^n-t_1^n|^{N-2}}
=o(\frac{\Lambda_2^n\Lambda_3^n}{|t_3^n-t_2^n|^{N-2}})$,
which is equivalent to
$\Lambda_1^n\Lambda_2^ng(t_1^n,t_2^n)=o(\Lambda_2^n\Lambda_3^ng(t_2^n,t_3^n))$.
Now \eqref{bou} implies $\Lambda_1^n\Lambda_2^ng(t_1^n,t_2^n)=o(1)$, hence
\eqref{geppy1} becomes
$$
(\beta_1^n+\beta_2^n)\Lambda_1^n\Lambda_4^ng(t_1^n,t_4^n)
= \beta_1^n+\beta_2^n+(\beta_1^n-\beta_2^n)\Lambda_1^n \Lambda_3^ng(t_1^n,t_3^n)+o(1)
\geq \beta_1^n+\beta_2^n+o(1)
$$
because $\beta_1^n>\beta_2^n$. Then $\Lambda_1^n\Lambda_4^ng(t_1^n,t_4^n)\geq c$, which implies
$g(t_1^n,t_4^n)\to +\infty$ by \eqref{casse2}. So, $t_4^n\to \bar t$ and then
$$
\frac{\Lambda_{1}^n\Lambda_{4}^n}{|t_1^n-t_4^n|^{N-1}}
\geq \sigma_N(N-2)\Lambda_{1}^n\Lambda_{4}^n\frac{g(t_1^n, t_4^n)}{|t_1^n-t_4^n|}
\to +\infty,
$$
in contradiction with $4\not \in I_1$.

An analogous argument applies when
\begin{equation*}
\Lambda_3^n,\,\Lambda_2^n\to 0,\quad \Lambda_1^n\geq c.
\end{equation*}
Finally, assume that
\begin{equation}\label{casse3}
\Lambda_2^n\to 0,\quad \Lambda_1^n,\,\Lambda_3^n\geq c.
\end{equation}
Then we obtain, using \eqref{sist-},
$$
\Lambda_1^n\Lambda_3^n \leq \frac{n}{ g(t_1^n,t_3^n)} \leq C n |t_1^n-t_3^n|^{N-2}
\leq Cn( |t_1^n-t_2^n|^{N-2}+ |t_2^n-t_3^n|^{N-2})
\leq Cn(\Lambda_1^n\Lambda_2^n+ \Lambda_2^n\Lambda_3^n)
$$
where the last inequality follows from \eqref{idde}. So, using \eqref{casse3}, we deduce
$c\leq \Lambda_1^n\Lambda_3\leq Cn \Lambda_2^n,$ by which  $\Lambda_2^n\geq \frac{c}{n}$.
Combining this with \eqref{elle4} and \eqref{casse3} we obtain
$$
\Lambda_1^n\cdot\Lambda_2^n\cdot\Lambda_3^n\cdot\Lambda_4^n\geq \frac{c}{n}.
$$
Finally \eqref{sist-}, \eqref{bou1} and \eqref{bou} imply
$$
\tilde{\Psi}^* = \tilde{\Psi} ( \bbm[\Lambda]_n,\bbm[t]_n)
= -\frac{n}{2}+O(1)-\log (\Lambda_1^n\Lambda_2^n\Lambda_3^n\Lambda_4^n)
\leq -\frac{n}{2}+O(1)+\log n\to -\infty
$$
in contradiction with the lower estimate \eqref{roxx}.

\bigskip

{\bf Case 5:  Conclusion.}

\bigskip

In order to not fall again in the cases 1-2, we assume:
$$
|\beta_1^n-\beta_2^n|\Lambda_1^n\Lambda_3^ng(t_1^n, t_3^n)\leq C,
 \quad  |\beta^n_1-\beta_2^n|\Lambda_2^n\Lambda_4^ng(t_2^n, t_4^n)\leq C.
$$
So, as in the cases 3 and 4 we immediately get \eqref{pooh1}--\eqref{bou} and, in particular,
$\Lambda_i^n\leq C$ for any $i=1,2,3,4$.
Moreover we may also assume
\begin{equation}\label{dindin}
\Lambda_i^n\geq c\quad \forall i=1,2,3,4.
\end{equation}
Indeed, assume for instance, that $\Lambda_1^n\to 0$. Then, by \eqref{motiv} for $i=1$  we
have that, either
\begin{equation}\label{conclu1}
(\Lambda_1^n)^2h(t_1^n,t_1^n)\geq  c,
\end{equation}
or
\begin{equation}\label{conclu2}
\exists j=2,3,4\ \hbox{ such that }\
|(-1)^{1+j}\beta_1^n-\beta_2^n| \Lambda_1^n\Lambda_j^ng(t_1^n,t_j^n)\geq c.
\end{equation}
If \eqref{conclu1} holds, then $h(t_1^n,t_1^n)\to +\infty$, which implies $\bar t_1=a$ or
$\bar t _1=b$ by \eqref{coercive}, and we are back in the case 3. On the other hand, if
\eqref{conclu2} holds, then, $g(t_1^n,t_j^n) \to +\infty$ for some $j\neq 1$, which implies
$\bar t_j=\bar t_1$, and we are either in the case 3 (if $\bar t_1=a,b$) or in case 4 (if
$\bar t_1\in (a,b)$). Finally \eqref{sist-}, \eqref{bou1}, \eqref{bou}, \eqref{dindin} imply
$$
\tilde{\Psi}^*=\tilde{\Psi} ( \bbm[\Lambda]_n,\bbm[t]_n)=-\frac{n}{2}+O(1)\to -\infty
$$
in contradiction with the lower estimate \eqref{roxx}.

\appendix
\renewcommand{\theequation}{\Alph{section}.\arabic{equation}}

\section{Some properties of the Green's function}
Let $\Omega$ be a bounded domain with a ${\cal C}^2$-boundary. We denote by $G(x,y)$ the
Green's function of $-\Delta$ on $\Omega$ under Dirichlet boundary conditions, and by $H(x,y)$
its regular part, as in the introduction. So $H$ satisfies
\begin{equation*}
\left\{
\begin{aligned}
&\Delta_yH(x,y)=0 &\hbox{ }&y\in\Omega,\\
&H(x,y)=\frac{1}{(N-2)\sigma_N|x-y|^{N-2}} &\hbox{ }&y\in\partial \Omega.
\end{aligned}
\right.
\end{equation*}
We recall that $H$ is a smooth function in $\Omega\times \Omega$; moreover $G$ and $H$ are
symmetric in $x$ and $y$ and $G, H>0$ in $\Omega\times\Omega$.

The diagonal $H(x,x)$ is called the Robin's function of the domain $\Omega$ and satisfies
\begin{equation}\label{coercive}
H(x,x)\to +\infty\quad \hbox{ as }{\rm{d}}(x):=\dist(x,\partial\O)\to 0.
\end{equation}
Let $H_0$ be the minimum value of the Robin's function:
$$H_0=\min_{\Omega}H(x,x)>0.$$
Recall that the Robin's function of a convex bounded domain is strictly convex (\cite{cara}).

We  need the following  result concerning the behavior of the regular part $H(x,y)$ near the
boundary. To this aim we fix $\delta>0$ sufficiently small such that the projection onto
$\partial \Omega$ is well defined in the region $\Omega_0:=\{x\in\Omega:{\rm d}(x)<\delta\}$;
we denote this projection by $p:\O_0 \to \partial\O$. It is of class ${\cal C}^1$ because
$\partial\Omega$ is of class ${\cal C}^2$. Moreover, for $x\in \Omega_0$, we write
$\bar x=2p(x)-x$ for the reflection of $x$ at $\partial \Omega$ and
$\nu_x=\frac{x-p(x)}{|x-p(x)|}$ for the inward unit normal at $p(x)$.

\begin{lemma}\label{robin}
Let $\Omega$ be a bounded domain with a ${\cal C}^2$-boundary. Then the following expansions
hold uniformly for $x\in \Omega_0$ and $y\in \Omega$:
$$
H(x,y)=\frac{1}{(N-2)\sigma_N|\bar x-y|^{N-2}}+O\bigg(\frac{{\rm{d}}(x)}{|\bar x-y|^{N-2}}\bigg),
$$
and
$$
\frac{\partial H}{\partial \nu_x}(x,y)
=\frac{1}{(N-2)\sigma_N}\frac{\partial}{\partial\nu_x}\bigg(\frac{1}{|\bar x-y|^{N-2}}\bigg)
  +O\bigg(\frac{1}{|\bar x-y|^{N-2}}\bigg).
$$
\end{lemma}

\begin{proof}
During the proof we will often use the symbols $c$, $C$ to denote different positive constants
depending only on $\O$. For any $x\in \Omega_0$ we introduce a diffeomorphism which straightens
the boundary near $p(x)$. Let $T_x$ be a rotation and translation of coordinates which maps
$p(x)$ to $0$ and the unit inward normal $\nu_x$ to the vector ${\bf e}_N:=(0,\ldots, 0,1)$.
Then $T_x (x) = (0,\ldots,0,{\rm{d}}(x))$, $T_x (\bar x) = (0,\ldots,0,-{\rm{d}}(x))$,
and in some neighborhood of $0$ the boundary $\partial (T_x \Omega)$ can be represented by
$$
z_N=\rho_x(z'), \quad  z'=(z_1,\ldots, z_{N-1});
$$
here $\rho_x$ is a $\cal C^2$ function satisfying $\rho_x(0)=0$ and $\nabla \rho_x(0)=0$.
Therefore we have
$$
|z_N|\leq C |z'|^2\quad \hbox{ on }\partial (T_x \Omega).
$$
First we prove the following estimate for the boundary points:
\begin{equation}\label{bouesti1}
\bigg|\frac{1}{|x-{y}|^{N-2}}-\frac{1}{|\bar x-y|^{N-2}}\bigg|
 \leq C\frac{ {\rm{d}}(x)}{|\bar x- y|^{N-2}}\quad \forall x\in\Omega_0,\;
\forall y\in \partial \Omega.
\end{equation}
In order to see this, we observe for $x\in\Omega_0$, $y\in\partial \Omega$, $z:=T_x(y)$, that
\begin{equation}\label{dile2}
\max\{{\rm{d}}(x),|z'|\} \le \min\{|x-y|,|\bar x-y|\},
\end{equation}
by which
\begin{equation}\label{lab}
\big||x-{y}|^{2}-|\bar x-y|^{2}\big|=4{\rm{d}}(x) z_N\leq C{\rm{d}}(x) |z'|^2
 \leq C{\rm{d}}(x) \min\{|\bar x-{y}|^2, |x-y|^2\}.
\end{equation}
The above inequality implies
\begin{equation}\label{dill}
c\leq \frac{|\bar x- y|}{|x-y|}\leq C\quad \forall x\in \Omega_0, \;
\forall y\in\partial \Omega.
\end{equation}
Taking into account that
$|a^m-b^m|\leq m|a-b|(a+b)^{m-1}$  for any $ a, b\geq 0$ and $m\geq 1$, we have
$$
\begin{aligned}
\bigg|\frac{1}{|x-y|^{N-2}}-\frac{1}{|\bar x-{y}|^{N-2}}\bigg|
 &=\bigg|\frac{|\bar{x}-y|^{2(N-2)}-|x-y|^{2(N-2)}}
              {(|\bar{x}-y|^{N-2}+|x-y|^{N-2})|{x}-y|^{N-2}|\bar{x}-y|^{N-2}}\bigg|\\
 &\leq(N-2)\frac{(|\bar x-{y}|^2+|x-y|^2)^{N-3}(|x-\bar{y}|^{2}-|x-y|^{2})}
                {(|\bar x-{y}|^{N-2}+|x-y|^{N-2})|x-y|^{N-2}|\bar x-{y}|^{N-2}}
\end{aligned}
$$
and \eqref{bouesti1} follows by using \eqref{lab} and \eqref{dill}. So, for any $x\in\Omega_0$,
the functions $H(x,y)-\frac{1}{\sigma_N(N-2)|\bar x-y|^{N-2}}$ and $\frac{1}{|\bar x-y|^{N-2}}$
are both harmonic in $\Omega$ in the variable $y$, and verify \eqref{bouesti1} on the boundary.
Then the maximum principle applies and gives
$$
\bigg|H(x,y)-\frac{1}{\sigma_N(N-2)|\bar x-y|^{N-2}}\bigg|
\leq C\frac{{\rm{d}}(x)}{|\bar x-y|^{N-2}}
\quad \forall x\in \Omega_0, \;\forall y\in \Omega.
$$
The first part of the thesis follows.

We go on with the normal derivative estimate. We claim the following estimate on the boundary:
\begin{equation}\label{bouesti2}
\bigg|\frac{\partial H}{\partial \nu_x}(x,y)
 -\frac{(\bar x-y)\cdot \nu_x}{\sigma_N|\bar x- y|^{N}}\bigg|
= \bigg|\frac{(y-x)\cdot \nu_x}{\sigma_N|x-y|^{N}}
  -\frac{(\bar x-y)\cdot \nu_x}{\sigma_N|\bar x- y|^{N}}\bigg|
\leq \frac{C}{|\bar x- y|^{N-2}}
\quad \forall x\in \Omega_0,\; \forall y\in\partial\Omega.
\end{equation}
Indeed, proceeding as for \eqref{bouesti1}  we have
\begin{equation}\label{above}
\bigg|\frac{1}{|x-y|^{N}}-\frac{1}{|\bar x- y|^{N}}\bigg|
 \leq C\frac{{\rm{d}}(x)}{|\bar x- y|^N}\leq \frac{C}{|\bar x- y|^{N-1}}
 \quad \forall x\in \Omega_0,\; \forall y\in \partial \Omega
\end{equation}
where the second inequality holds since ${\rm{d}}(x)\leq |\bar x-y|$ by \eqref{dile2}.
Moreover, for $x\in\Omega_0$, $y\in\partial \Omega$, $z:=T_x(y)$,
\begin{equation}\label{dile3}
\begin{aligned}
|(y-x)\cdot \nu_x-(\bar x-y)\cdot \nu_x|
 &=|(z-{\rm{d}}(x) {\bf e}_N)\cdot {\bf e}_N-(-{\rm{d}}(x) {\bf e}_N-z)\cdot {\bf e}_N|\\
 &=2|z_N|\leq C|z'|^2\leq C|\bar x-{y}|^2
\end{aligned}
\end{equation}
where for the last inequality we have used \eqref{dile2}.
Thus we obtain for $x\in \Omega_0$ and $y\in \partial \Omega$:
$$
\bigg|\frac{(y-x)\nu_x}{|x-y|^{N}}-\frac{( \bar x-y)\nu_x}{|\bar x- y|^{N}}\bigg|
 \leq |(y-x) \nu_x|\bigg|\frac{1}{|x-y|^{N}}-\frac{1}{|\bar x- y|^{N}}\bigg|
       +\frac{|(y-x)\nu_x-(\bar x- y) \nu_x|}{|\bar x- y|^{N}}
$$
and \eqref{bouesti2} follows from \eqref{dill}, \eqref{above}, \eqref{dile3}.
Now, for $x\in \Omega_0$ fixed, the functions
$\frac{\partial H}{\partial \nu_x}(x,y)-\frac{(\bar x-y)\cdot \nu_x}{\sigma_N|\bar x- y|^N}$
and $\frac{1}{|\bar x-y|^{N-2}}$ are harmonic in $\Omega$ with respect to the variable $y$,
and verify \eqref{bouesti2} on the boundary. The maximum principle applies and gives
$$
\bigg|\frac{\partial H}{\partial \nu_x}(x,y)
 -\frac{(\bar x-y)\cdot \nu_x}{\sigma_N|\bar x- y|^N}\bigg|
 \leq \frac{C}{|\bar x-y|^{N-2}}\quad \forall x\in\Omega, \;\forall y\in \Omega_0.
$$
In order to conclude observe that $\frac{\partial \bar x}{\partial \nu_x }=-\nu_x$, because
$\frac{\partial  p}{\partial \nu_x}(x)=0$ for any $x\in\Omega_0$, so that
$$
\frac{\partial }{\partial \nu_x }\bigg(\frac{1}{(N-2)|\bar x-y|^{N-2}}\bigg)
 =\frac{(\bar x-y)\cdot \nu_x}{|\bar x- y|^N}\quad \forall x\in\Omega_0,\;\forall y\in\Omega.
$$
\end{proof}

We conclude this section with the following lemma which is concerned with the behaviour of
$G(\cdot,y)$ along half-lines through the domain starting from $y$. This implies
\eqref{prop2} for convex domains.

\begin{lemma}\label{lem}
Let $\Omega$ be a convex and bounded domain with a smooth boundary. Then for any
$x,y\in\Omega$, $x\neq y$,  we have
$$
(x-y)\cdot\nabla_xG(x,y)<0.
$$
\end{lemma}

\begin{proof} We use Lemma 3.1 in \cite{grotak} which states that if $\Omega$ is a smooth and
bounded domain in $\R^N$, then, for any $P\in\Omega$, $A,B\in\Omega$, $A\neq B$,
$$
-\int_{\partial\Omega}(x-P)\cdot \nu_x
 \frac{\partial G(x,A)}{\partial \nu_x}\frac{\partial G(x,B)}{\partial \nu_x}ds
=(2-N)G(A,B)+(P-A)\nabla_x G(A,B)+(P-B)\nabla_xG(B,A),
$$
where $\nu_x$ is the unit inner normal at $x\in\partial \Omega$. Now assume that $\Omega$ is
convex and take $P=B$. We deduce
$$
(B-A)\nabla_x G(A,B)
 =-\int_{\partial\Omega}(x-B)\cdot \nu_x\frac{\partial G(x,A)}{\partial \nu_x}
   \frac{\partial G(x,B)}{\partial \nu_x}ds + (N-2)G(A,B)
$$
which is strictly positive because $(x-B)\cdot\nu_x<0$ for any $x\in\partial \Omega$ by the
convexity of $\Omega$, and because
$\frac{\partial G(x,A)}{\partial \nu_x},\, \frac{\partial G(x,B)}{\partial \nu_x}>0$ on
$\partial\Omega$.
\end{proof}

\end{document}